\documentclass[11pt, a4paper]{amsart}

\usepackage{gfsartemisia}
\usepackage[T1]{fontenc}
\usepackage{amsfonts}
\usepackage{amssymb}
\usepackage{amsmath, amsthm,color}
\usepackage{hyperref}
\usepackage{pdfsync}
\usepackage{bbm, dsfont}
\usepackage[margin=0.96in]{geometry}
\usepackage{enumerate}   
\usepackage[shortlabels]{enumitem}
\usepackage{color}
\usepackage{multicol}
\usepackage{graphicx}

\usepackage{amsrefs}

\usepackage[mathscr]{eucal}
\usepackage{mathrsfs}

\newcommand{\La}{\langle}
\newcommand{\Ra}{\rangle}

\def\supp{\operatorname{supp}}

\newcommand{\unit}{1\!\!1}

\newcommand{\bE}{\mathbb{E}}
\newcommand{\bR}{\mathbb{R}}

\newcommand{\bN}{\mathbb{N}}

\newcommand{\cE}{\mathcal{E}}

\newcommand{\cD}{\mathcal{D}}

\newcommand{\cS}{\mathcal{S}}
\newcommand{\cA}{\mathcal{A}}
\newcommand{\cF}{\mathcal{F}}
\newcommand{\cG}{\mathcal{G}}

\newtheorem{theorem}{Theorem}[section]
\newtheorem{corollary}[theorem]{Corollary}
\newtheorem{lemma}[theorem]{Lemma}
\newtheorem{prop}[theorem]{Proposition}

\theoremstyle{remark}
\newtheorem{remark}{Remark}[section]

\numberwithin{equation}{section}

\begin{document}

\title[Paraproducts, Bloom BMO and Sparse BMO Functions]{Paraproducts, Bloom BMO and Sparse BMO Functions}

\author[Valentia Fragkiadaki]{Valentia Fragkiadaki}
\address{Department of Mathematics, Texas A\&M University, College Station, TX 77843, USA}
\email{valeria96@tamu.edu}

\author[Irina Holmes Fay]{Irina Holmes Fay}
\address{Department of Mathematics, Texas A\&M University, College Station, TX 77843, USA}
\email{irinaholmes@tamu.edu}
\thanks{I. Holmes Fay is supported by Simons Foundation: Mathematics and Physical Sciences-Collaboration Grants for Mathematicians, Award number 853930.}

\makeatletter
\@namedef{subjclassname@2010}{
  \textup{2010} Mathematics Subject Classification}
\makeatother

\begin{abstract} 
We address $L^p(\mu)\rightarrow L^p(\lambda)$ bounds for paraproducts in the Bloom setting. We introduce certain ``sparse BMO'' functions associated with sparse collections with no infinitely increasing chains, and use these to express sparse operators as sums of paraproducts and martingale transforms -- essentially, as Haar multipliers -- as well as to obtain an equivalence of norms between sparse operators $\cA_\cS$ and compositions of paraproducts $\Pi^*_a\Pi_b$.
\end{abstract}
\maketitle

In 1985, Steven Bloom proved \cite{Bloom} that the commutator $[b,H]f = b\cdot Hf - H(b\cdot f)$, where $H$ is the Hilbert transform, is bounded $L^p(\mu)\rightarrow L^p(\lambda)$, where $\mu,\lambda$ are two $A_p$ weights ($1<p<\infty$), if and only if $b$ is in a weighted BMO space determined by the two weights $\mu$ and $\lambda$, namely $b\in BMO(\nu)$, where $\nu:=\mu^{1/p}\lambda^{-1/p}$ and
	$$
	\|b\|_{BMO(\nu)} := \sup_Q \frac{1}{\nu(Q)}\int_Q|b(x)-\La b\Ra_Q|\,dx.
	$$
In \cite{HLW} this result was extended to commutators $[b, T]$ in $\bR^n$ with Calder\'{o}n-Zygmund operators $T$. Soon after, \cite{Lerner} gave a different proof which yielded a quantitative result for the upper bound:
	\begin{equation}
	\label{E:Lerner}
	\|[b, T] :L^p(\mu) \rightarrow L^p(\lambda)\| \lesssim \|b\|_{BMO(\nu)} \bigg([\mu]_{A_p} [\lambda]_{A_p}\bigg)^{\max\left(1, \frac{1}{p-1}\right)}.
	\end{equation}
The proof in \cite{HLW} took the route of Hyt\"{o}nen's representation theorem (the $\bR^n$, Calder\'{o}n-Zygmund operator generalization of Petermichl's result \cite{Stef} on the Hilbert transform), and relied heavily on paraproduct decompositions. The proof in \cite{Lerner} used sparse operators and Lerner's median inequalities to obtain directly a sparse domination result for the commutator $[b, T]$ itself, avoiding paraproducts althogehter. 

This paper addresses $L^p(\mu)\rightarrow L^p(\lambda)$ bounds for the paraproducts. Based on the one-weight situation, we suspect that these bounds should be smaller than the ones for commutators: in the one-weight case
	$$
	\|[b, H]:L^p(w)\rightarrow L^p(w)\| \lesssim \|b\|_{BMO} [w]_{A_p}^{2\max\left(1, \frac{1}{p-1}\right)}
	$$
and
	$$
	\|\Pi_b :L^p(w)\rightarrow L^p(w)\| \lesssim \|b\|_{BMO} [w]_{A_p}^{\max\left(1, \frac{1}{p-1}\right)},
	$$
are both known to be sharp -- see \cite{PereyraSparse,Chung} and the references therein -- (where throughout this paper $A\lesssim B$ is used to mean $A\leq C(n)B$, with a constant depending on the dimension and maybe other quantities such as $p$ or Carleson constants $\Lambda$ of sparse collections, but in any case not depending on any $A_p$ characteristics of the weights involved). In the two-weight Bloom situation, we show in Theorem \ref{T:ParaBd} that
	$$
	\|\Pi_b : L^p(\mu)\rightarrow L^p(\lambda)\| \lesssim \|b\|_{BMO(\nu)}[\mu']_{A_{p'}}[\lambda]_{A_p} =
	\|b\|_{BMO(\nu)} [\mu]_{A_p}^{\frac{1}{p-1}}[\lambda]_{A_p}
	$$
We do not know if this bound is sharp, and this is subject to future investigations -- but the bound is smaller than the one in \eqref{E:Lerner}. In fact, it is strictly smaller with the exception of $p=2$, when both bounds are $[\mu]_{A_2}[\lambda]_{A_2}$. We can however show that our bound is sharp in one particular instance, namely when $\mu=w$ and $\lambda=w^{-1}$ for some $A_2$ weight $w$. We show this in Section \ref{Ss:sharp} via an appeal to the one-weight linear $A_2$ bound for the dyadic square function.

Obviously this bound does not recover the one-weight situation: letting $\mu=\lambda=w$ for some $w\in A_2$, $\nu=1$ and our bound would give
	$$
	\|\Pi_b:L^2(w)\rightarrow L^2(w)\|\lesssim \|b\|_{BMO}[w]^2,
	$$
when we know that the optimal bound is linear in the $A_2$ characteristic. If the optimal Bloom paraproduct bound is to recover this one-weight situation, we suspect it would need a dependency on $[\nu]_{A_2}$ -- as it would need to somehow account for the case $\mu=\lambda$, or $\nu=1$.

The proof of the Bloom paraproduct bound above relies on dominating the paraproduct by a ``Bloom sparse operator''
	$
	\cA_\cS^\nu f:= \sum_{Q\in\cS}\La \nu\Ra_Q\La f\Ra_Q \unit_Q,
	$
where $\cS$ is a sparse collection, and proving that $\cA_\cS^\nu$ satisfies the bound $[\mu]_{A_p}^{1/(p-1)}[\lambda]_{A_p}$ above. We do this in Theorem \ref{T:Bspbd}. The domination of the paraproduct is treated in Section \ref{S:ParaDom}.

Before all this however, we consider in Section \ref{S:SparseBMO} a special type of sparse collections, $\Upsilon^\cD(\bR^n)$, which are sparse collections with no ``infinitely increasing chains'' (a terminology borrowed from \cite{HytonenLectures}). We see that any such collection can be associated with a BMO function
	$$
	b_\cS := \sum_{Q\in\cS}\unit_Q,
	$$
which satisfies $\|b_\cS\|_{BMO}\leq\Lambda$, where $\Lambda$ is the Carleson constant of $\cS$ (we show this in Appendix \ref{A:my-sparse}). Once we have a BMO function, we can immediately talk about paraproducts with symbol $b_\cS$. In fact, we see in Section \ref{sS:sparse=sum} that these functions allow us to express any sparse operator $\cA_\cS$, $\cS\in\Upsilon^\cD(\bR^n)$, as a sum of paraproducts and a martingale transform:
	$$
	\cA_\cS f = \Pi_{b_\cS}f + \Pi^*_{b_\cS}f + T_{\tau_\cS}f,
	$$
where $T_{\tau_\cS}$ is a martingale transform:
	$$
	T_{\tau_\cS} = \sum_{J\in\cD} (\tau_\cS)_J (f, h_J)h_J, \text{ where } (\tau_\cS)_J := \frac{1}{|J|}\sum_{I\in\cS, I\subsetneq J}|I| \leq \Lambda, \:\:\forall J\in\cD.
	$$
As discussed in Section \ref{sS:sparse=sum}, this gives us an upper bound for norms of sparse operators in terms of norms of paraproducts and martingale transforms, and in fact the equivalence
	$$
	\sup_{\cS\in\Upsilon^\cD}\|\cA_\cS\|_{L^p(w)\rightarrow L^p(w)} \simeq_{n,p,\Lambda} 
		\sup_{b\in BMO^\cD}\frac{\|\Pi_b\|_{L^p(w)\rightarrow L^p(w)}}{\|b\|_{BMO^\cD}}
		+ \sup_{b\in BMO^\cD}\frac{\|\Pi^*_b\|_{L^p(w)\rightarrow L^p(w)}}{\|b\|_{BMO^\cD}}
		+ \sup_{\tau\in\ell^\infty}\frac{\|T_{\tau}\|_{L^p(w)\rightarrow L^p(w)}}{\|\tau\|_{\infty}}.
	$$

The process used to obtain the BMO function $b_\cS$ associated with $\cS$ also works with weights, and obtaining a function in weighted BMO spaces associated with $\cS\in\Upsilon^\cD(\bR^n)$: if $w\in A_p$, the function
	$$
	b_\cS^w:=\sum_{Q\in\cS}\La w\Ra_Q \unit_Q
	$$
is in $BMO^\cD(w)$, with $\|b_\cS^w\|_{BMO^\cD(w)} \leq 2[w]_{A_p}\Lambda^p$. Repeating the process above, we try to express $\cA_\cS$ as a sum of the paraproducts associated with $b_\cS^w$ and a martingale transform -- but we discover instead the operator
	$$
	\cA_\cS^w f:= \sum_{Q\in\cS}\La w\Ra_Q\La f\Ra_Q\unit_Q,
	$$
and its decomposition as
	$$
	\cA_\cS^w f = \Pi_{b_\cS^w}f + \Pi^*_{b_\cS^w}f + T_{\tau_\cS^w}f,	
	$$
detailed in Proposition \ref{P:SPM}.

While it would be interesting if the paraproducts and the martingale transform could somehow be ``separated'' above, giving an independent proof that these operators have the same dependency on $[w]_{A_p}$ by showing each is equivalent to norms of $\cA_\cS$, we are able to show that norms of sparse operators are equivalent to certain compositions of paraproducts. In Section \ref{sS:ParaComp}, we see that
	$$\cA_\cS \simeq \Pi^*_{\widetilde{b}_\cS}\Pi_{\widetilde{b}_\cS},$$
where $\widetilde{b}_\cS$ is another BMO function we can easily associate with $\cS$: 
	$$\widetilde{b}_\cS := \sum_{Q\in\cS}\sqrt{|Q|}h_Q.$$
This provides an upper bound:
	$$
	\sup_{\substack{\cS\in\Upsilon^\cD(\bR^n) \\ \Lambda_{(\cS)}=\Lambda}} \frac{\|\cA_\cS:L^p(w)\rightarrow L^p(w)\|}{\Lambda}
	\leq \sup_{a,b\in BMO^\cD} \frac{\|\Pi^*_a\Pi_b :L^p(w)\rightarrow L^p(w)\|}{\|a\|_{BMO^\cD}\|b\|_{BMO^\cD}}.
	$$	
For the other direction, we show in Appendix \ref{A:ParaComp} -- using a bilinear form argument -- that for all Bloom weights $\mu,\lambda,\nu$, BMO functions $a\in BMO^\cD$, $b\in BMO^\cD(\nu)$, and $\Lambda>1$, 
	$$
	\|\Pi_a^*\Pi_b :L^p(\mu)\rightarrow L^p(\lambda)\| \leq C(n) \|a\|_{BMO^\cD}\|b\|_{BMO^\cD(\nu)} 
	\sup_{\substack{\cS\in\Upsilon^\cD(\bR^n)\\ \Lambda_{(\cS)}=\Lambda}} \left(\frac{\Lambda}{\Lambda-1}\right)^3 
	\|\cA_\cS^\nu:L^p(\mu) \rightarrow L^p(\lambda)\|.
	$$
Note that taking $\mu=\lambda=w$ above, for some $w\in A_p$, we have the one-weight result
	$$
	\|\Pi_a^*\Pi_b:L^p(w)\rightarrow L^p(w)\|\lesssim \|a\|_{BMO^\cD}\|b\|_{BMO^\cD}[w]_{A_p}^{\max\left(1, \frac{1}{p-1}\right)}.
	$$
Moreover, we obtain the equivalence of norms
	$$
	\sup_{\cS\in\Upsilon^\cD(\bR^n)} \|\cA_\cS:L^p(w)\rightarrow L^p(w)\| \simeq_{\Lambda, p, n} \sup_{a,b\in BMO^\cD(\bR^n)} 
	\frac{\|\Pi^*_a\Pi_b :L^p(w)\rightarrow L^p(w)\|}{\|a\|_{BMO^\cD}\|b\|_{BMO^\cD}}.
	$$

Section \ref{S:ParaDom} gives a proof of a pointwise domination of paraproducts by sparse operators.
It relies on first proving certain local pointwise domination results, which are then applied to $BMO_\cD(w)$ functions with finite Haar expansion, and extending to the general case. So this argument works whenever $\Pi_b$ acts between $L^p$ spaces where the Haar system is an unconditional basis -- Lebesgue measure or $A_p$ weights. The argument also works with the weighted BMO norm, 
	$$
	\|b\|_{BMO^\cD(w)} := \sup_{Q\in\cD} \frac{1}{w(Q)} \int_Q |b - \La b\Ra_Q|\,dx,
	$$
defined in terms of an $L^1(dx)$ quantity -- the Haar system is not unconditional in $L^1(dx)$, but we can choose an ordering of the Haar system that ensures convergence in $L^1(dx)$.
The choice to work with $b$ rather than compactly supported $f$ is motivated by the desire to obtain domination by sparse operators with no infinitely increasing chains. Specifically, we work with restricted paraproducts:
	$$
	\Pi_{b, Q_0} f(x) := \sum_{Q\in Q_0} (b, h_Q) \La f\Ra_Q h_Q(x), \:\:\forall Q_0\in\cD,
	$$
and construct a sparse collection $\cS(Q_0) \subset \cD(Q_0)$ which ``ends'' at $Q_0$, and such that $\cA_\cS^w f$ pointwise dominates $\Pi_{b, Q_0} f$ on $Q_0$. Since the Haar expansion of $b$ effectively dictates the Haar expansion of $\Pi_b$ (as well as $\Pi_b^*$ and $\Gamma_b$), this will lead from finite Haar expansion $b$'s to collections in $\Upsilon_\cD(\bR^n)$.

\vspace{0.1in}

The second author deeply thanks Cristina Pereyra for several conversations about this work, and for her general support.


\section{Setup and Notations}

\subsection{Dyadic Grids.} By a dyadic grid $\cD$ on $\bR^n$ we mean a collection of cubes $Q\subset\bR^n$ that satisfies:
	\begin{itemize}
	\item Every $Q\in\cD$ has side length $2^k$ for some $k\in\mathbb{Z}$: $\ell(Q) = 2^{k}$;
	\item For a fixed $k_0\in\mathbb{Z}$, the collection $\{Q\in\cD:\ell(Q)=2^{k_0}\}$ forms a partition on $\bR^n$;
	\item For every $P,Q\in\cD$, the intersection $P\cap Q$ is one of $\{P,Q,\emptyset\}$. In other words, two dyadic cubes intersect each other if and only if one contains the other.
	\end{itemize}
For example, the standard dyadic grid on $\bR^n$ is:
	$$
	\cD_0 := \{2^{-k}([0,1)^n+m): k\in\mathbb{Z}, m\in\mathbb{Z}^n\}.
	$$
We assume such a collection $\cD$ is fixed throughout the paper. For every $Q\in\cD$ and positive integer $k\geq 1$ we let $Q^{(k)}$ denote the $k^{th}$ dyadic ancestor of $Q$ in $\cD$, i.e. the unique $R\in\cD$ such that $R\supset Q$ and $\ell(R)=2^k \ell(Q)$. Given $Q_0\in\cD$, we let $\cD(Q_0)$ denote the collection of dyadic subcubes of $Q_0$:
	$$
	\cD(Q_0) :=\{Q\in\cD: Q\subset Q_0\}.
	$$

\subsection{Haar Functions.} Given a dyadic grid $\cD$ on $\bR$, we associate to each $I\in\cD$ the cancellative Haar function
$h_I:=h_I^0 = \frac{1}{\sqrt{|I|}}(\unit_{I_+}-\unit_{I_-})$, where $I_+$ and $I_-$ are the right and left halves of $I$, respectively. The non-cancellative Haar function is $h_I^1 := \frac{1}{\sqrt{|I|}}\unit_I$. The cancellative Haar functions $\{h_I\}_{I\in\cD}$ form an orthonormal basis for $L^2(\bR, dx)$, and an unconditional basis for $L^p(\bR)$, $1<p<\infty$. Throughout this paper we let $(\cdot,\cdot)$ denote inner product in $L^2(\,dx)$, so we write for example
	$$
	f = \sum_{I\in\cD}(f, h_I)h_I,
	$$
where $(f, h_I)=\int f h_I\,dx$ is the Haar coefficient of $f$ corresponding to $I$.

In $\bR^n$, we have $2^n-1$ cancellative Haar functions and one non-cancellative: for every dyadic cube $Q = I_1\times I_2\times\ldots I_n$, where every $I_k\in\cD$ is a dyadic interval with common length $|I_k|=\ell(Q)$, we let
	$$
	h_Q^\epsilon(x) := h_{I_1\times\ldots\times I_n}^{(\epsilon_1, \ldots, \epsilon_n)}(x_1,\ldots, x_n) = \prod_{k=1}^n h_{I_k}^{\epsilon_k}(x_k),
	$$
where $\epsilon_k\in\{0,1\}$ for all $k$, and $\epsilon=(\epsilon_1,\ldots, \epsilon_n)$ is known as the signature of $h_Q^\epsilon$. The function
$h_Q^\epsilon$ is cancellative except in one case, when $\epsilon\equiv 1$. As in $\bR$, the cancellative Haar functions form an orthonormal basis for $L^2(\bR^n, dx)$, and an unconditional basis for $L^p(\bR^n, dx)$, $1<p<\infty$. We often write
	$$
	f = \sum_{Q\in\cD}(f, h_Q)h_Q
	$$
to mean
	$$
	f = \sum_{\substack{Q\in\cD,\\ \epsilon\not\equiv 1}}(f, h_Q^\epsilon)h_Q^\epsilon,
	$$
omitting the signatures, and understanding that $h_Q$ always refers to a cancellative Haar function. There is really only one instance for us where the signatures matter, and that is in the definition of the paraproduct $\Gamma_b$ in $\bR^n$, $n>1$.

Note that whenever $P\subsetneq Q$ for some dyadic cubes $P, Q$, the Haar function $h_Q$ will be constant on $P$. We denote this constant by 
	$$
	h_Q(P) := \text{ the constant value } h_Q \text{ takes on } P\subsetneq Q.
	$$
It is easy to show that
	$$
	\La f\Ra_{Q}=\sum_{R\supsetneq Q}(f, h_R)h_R(Q), \:\:\forall Q\in\cD,
	$$
where throughout the paper
	$$
	\La f\Ra_Q := \frac{1}{|Q|}\int_Q f\,dx,
	$$
denotes average over $Q$, and sums such as $\sum_{P\subset Q}$ or $\sum_{R\supset Q}$ are understood to be over dyadic cubes.

\subsection{$A_p$ weights.} A weight is a locally integrable, a.e. positive function $w(x)$ on $\bR^n$. Any such weight immediately gives a measure on $\bR^n$ via $dw := w(x)dx$ and 
	$$
	\int f \,dw := \int f(x) w(x) \,dx
	$$
yields the obvious $L^p$-spaces associated with the measure $w$. We denote these spaces by $L^p(w)$.
Given $1<p<\infty$, we say $w\in A_p$ if
	$$
	[w]_{A_p} := \sup_Q \La w\Ra_Q \La w'\Ra_Q^{p-1} < \infty,
	$$
where the supremum is over cubes $Q\subset \bR^n$, $p'$ denotes the H\"{o}lder conjugate of $p$:
	$$
	\frac{1}{p}+\frac{1}{p'}=1,
	$$
and 
	$$
	w' := w^{1-p'}=w^{-p'/p}.
	$$
In fact, $w\in A_p$ if and only if  the conjugate weight $w'$ is in $A_{p'}$, with
	$$
	[w']_{A_{p'}} = [w]_{A_p}^{\frac{1}{p-1}}.
	$$
We restrict our attention to dyadic $A_p$ weights, denoted $A_p^\cD$, and defined in the same way except the supremum is only over dyadic cubes $Q\in\cD$.
Sometimes we use the standard $L^p$-duality $(L^p(w))^*=L^{p'}(w)$ with inner product $(\cdot, \cdot)_{dw}$, and other times we think of 
$(L^p(w))^*\simeq L^{p'}(w')$ with regular Lebesgue inner product $(\cdot,\cdot)$.
We refer the reader to Chapter 9 of \cite{Grafakos} for a thorough treatment of $A_p$ weights.

\subsection{Paraproducts and BMO} 

We say $b\in BMO(\bR^n)$ if
	$$
	\|b\|_{BMO(\bR^n)} := \sup_Q \frac{1}{|Q|}\int_Q|b(x)-\La b\Ra_Q|\,dx <\infty,
	$$
where the supremum is over cubes $Q\subset \bR^n$.
Given a weight $w$ on $\bR^n$, we say $b\in BMO(w)$ is in the weighted BMO space $BMO(w)$ if
	$$
	\|b\|_{BMO(w)} := \sup_Q \frac{1}{w(Q)} \int_Q|b(x)-\La b\Ra_Q|\,dx <\infty.
	$$
We similarly restrict our attention to dyadic BMO spaces, $BMO^\cD$ and $BMO^\cD(w)$ for the weighted version, both defined in the same way except the supremum is over dyadic cubes $Q\in\cD$.

In $\bR$, we have two paraproducts:
	\begin{eqnarray*}
	\Pi_b f(x) &:=& \sum_{I\in\cD}(b, h_I)\La f\Ra_Ih_I(x)\\
	\Pi^*_b f(x) &:=& \sum_{I\in\cD}(b, h_I)(f, h_I)\frac{\unit_I(x)}{|I|}.
	\end{eqnarray*}
They have the property that
	$$
	bf = \Pi_b f + \Pi^*_b f + \Pi_f b,
	$$
and their boundedness is usually characterized by some $BMO$-type norm of the symbol $b$. 

In $\bR^n$ we have three paraproducts:
	\begin{eqnarray*}
	\Pi_b f(x) &:=& \sum_{Q\in\cD}(b, h_Q)\La f\Ra_Qh_Q(x)\\
	\Pi^*_b f(x) &:=& \sum_{Q\in\cD}(b, h_Q)(f, h_Q)\frac{\unit_Q(x)}{|Q|}\\
	\Gamma_b f(x) &:=& \sum_{Q\in\cD}\sum_{\epsilon,\eta\not\equiv 1; \epsilon\neq\eta} (b, h_Q^\epsilon)(f, h_Q^\eta) \frac{1}{\sqrt{|Q|}}h_Q^{\epsilon+\eta}.
	\end{eqnarray*} 
$\Pi_b$ and $\Pi_b^*$ are adjoints in $L^2(\bR^n)$, and $\Gamma_b$ is self-adjoint. Generally, in the $L^p$-situation, we still have
	$$
	(\Pi_b f, g) = (f, \Pi_b^* g),
	$$
so if we think of $\Pi_b:L^p(\mu)\rightarrow L^p(\lambda)$ for two $A_p$ weights $\mu,\lambda$, its adjoint is $\Pi_b^*:L^{p'}(\lambda')\rightarrow L^{p'}(\mu')$ -- where we are thinking of Banach space duality in terms of $(L^p(\mu))^*\simeq L^{p'}(\mu')$ and $(L^p(\lambda))^*\simeq L^{p'}(\lambda')$, both with regular Lebesgue inner product $(\cdot,\cdot)$.

\section{Sparse BMO Functions}
\label{S:SparseBMO}
\subsection{Sparse Families.} Let $0 < \eta < 1$. A collection $\cS \subset \cD$ is said to be \textit{$\eta$-sparse} if for every $Q\in\cS$ there is a measurable subset $E_Q \subset Q$ such that the sets $\{E_Q\}_{Q\in\cS}$ are pairwise disjoint, and satisfy $|E_Q| \geq \eta |Q|$ for all $Q\in\cS$.

Let $\Lambda >1$. A family $\cS\subset\cD$ is said to be \textit{$\Lambda$-Carleson} if
	$$
	\sum_{P\in\cS, P \subset Q} |P| \leq \Lambda |Q|, \:\: \forall Q\in\cD.
	$$
It is easy to see that it suffices to impose this condition only on $Q\in\cS$. It is also easy to see that any $\eta$-sparse collection is $1/\eta$-Carleson. Far less obvious is the remarkable property that any $\Lambda$-Carleson collection is $1/\Lambda$-sparse, which is proved in the now classic work \cite{LernerNazarov}. 

A special type of sparse collection which appears most frequently in practice is defined in terms of so-called ``$\cS$-children.''
Suppose a family $\cS\subset\cD$ has the property that
	$$
	\sum_{P\in\text{ch}_\cS(Q)} |P| \leq \alpha |Q|, \forall Q\in\cS,
	$$
where $\alpha \in (0, 1)$ and $\text{ch}_\cS(Q)$, the $\cS$-children of $Q$, is the collection of maximal $P\in\cS$ such that $P \subsetneq Q$. Then $\cS$ is $(1-\alpha)$-sparse: let
	$$E_Q := Q \setminus \bigcup_{P\in\text{ch}_\cS(Q)}P,$$
which are clearly pairwise disjoint, and satisfy $|E_Q| \geq (1-\alpha)|Q|$.

A collection that is sparse with respect to Lebesgue measure is also sparse with respect to any $A_p$ measure $w$. 
Recall that (see \cite{Grafakos}, Proposition 9.1.5) an equivalent definition for $[w]_{A_p}$ is
	$$
	[w]_{A_p} = \sup_{Q\in\cD} \sup_{\substack{f \in L^p(Q, w) \\ |Q\cap\{|f|=0\}|=0}} \frac{\La|f|\Ra_Q^p}{\bE_Q^w (|f|^p)},
	$$
where
	$$
	\bE_Q^w f:= \frac{1}{w(Q)}\int_Q f\,dw.
	$$
Taking $f = \unit_A$ above, for some measurable subset $A$ of a fixed dyadic cube $Q$, we get
	$$
	\bigg( \frac{|A|}{|Q|}\bigg)^p \leq [w]_{A_p} \frac{w(A)}{w(Q)}, \:\:\forall A\subset Q, Q\in\cD.
	$$
So, say $\cS$ is $\eta$-sparse with pairwise disjoint $\{E_Q\}_{Q\in\cS}$ subsets $E_Q\subset Q$ and $|E_Q| \geq \eta|Q|$. Then
	$$
	\eta^p \leq \bigg( \frac{|E_Q|}{|Q|}\bigg)^p \leq [w]_{A_p} \frac{w(E_Q)}{w(Q)},
	$$ 
and
	\begin{equation}
	\label{E:w-sparse}
	w(Q) \leq \frac{1}{\eta^p} [w]_{A_p} w(E_Q), \:\:\forall Q\in\cS.
	\end{equation}

\vspace{0.1in}
\begin{center}
$\ast$
\end{center}
\vspace{0.1in}

\subsection{Sparse BMO Functions.} 
\label{Ss:SparseBMO}
We borrow the following terminology from \cite{HytonenLectures}: we say a collection $\cS\subset\cD$ has an \textit{infinitely increasing chain} if there exist $\{Q_K\}_{k\in\mathbb{N}}$, $Q_k \in\cS$, such that $Q_k \subsetneq Q_{k+1}$, for all $k\in\mathbb{N}$. The following Lemma is also found in \cite{HytonenLectures}:

\begin{lemma}
If a collection $\cS\subset\cD$ has no infinitely increasing chains, then every $Q\in\cS$ is contained in a maximal $Q^*\in\cS$ -- in the sense that there exists no $R\in\cS$ such that $R\supsetneq Q$. Any two maximal $P^*, Q^*$ elements of $\cS$ are disjoint.
\end{lemma}

These types of collections will be important for us, so we let 
	$$\Upsilon^\cD(\bR^n)$$ 
denote the set of all sparse collections in $\cD$ which have no infinitely increasing chains.

\begin{lemma}
\label{L:mzero}
Let $\cS\in\Upsilon^\cD(\bR^n)$ be a sparse collection with no infinitely increasing chains. Then the set of points contained in infinitely many elements of $\cS$ has measure $0$.
\end{lemma}

\begin{proof}
Let $\cS^*$ denote the collection of maximal elements of $\cS$. Since $\cS\in\Upsilon^\cD(\bR^n)$, every $Q\in\cS$ is contained in a unique $Q^*\in\cS^*$. Any $x$ which belongs to infinitely may elements of $\cS$ must then belong to an infinitely decreasing chain 
	$$
	x \in \ldots \subsetneq Q_k \subsetneq \ldots \subsetneq Q_2 \subsetneq Q_1 = Q^*
	$$
terminating at some maximal $Q^*\in\cS^*$. Fix any such chain and let $A$ be the set of points contained in all $Q_k$, that is $A = \bigcap_{k=1}^\infty Q_k$. Then for any $k\in\bN$:
	$$
	k|A| \leq \sum_{i=1}^k |Q_i| \leq \Lambda |Q^*|,
	$$
where $\Lambda$ is the Carleson constant of $\cS$. So $|A|\leq\frac{1}{k} \Lambda|Q^*|$ for all $k\in\bN$, and then $|A|=0$.

Alternatively, since $\{Q_k\}$ is a decreasing nest of sets, $|A|=\lim_{k\rightarrow\infty}|Q_k|$, and $\lim_{k\rightarrow\infty}|Q_k| = 0$ because the series
	$$
	\sum_{k=1}^\infty |Q_k| \leq \sum_{Q\in\cS, Q\subset Q^*}|Q| \leq \Lambda |Q^*|
	$$
converges.
\end{proof}

\vspace{0.1in}
\begin{center}
$\ast$
\end{center}
\vspace{0.1in}

The lemma above ensures that the following definition is sound: with every sparse collection $\cS\in\Upsilon^\cD(\bR^n)$ with no infinitely increasing chains
 we associate the function
 	$$
	b_\cS := \sum_{Q\in\cS}\unit_Q.
	$$
By Lemma \ref{L:mzero} we know that $b_{\cS}$ is almost everywhere finite: if $x$ is contained in infinitely many elements of $\cS$, then $b_\cS(x)=\infty$, but this can only happen on a set of measure zero.

Note also that $b_\cS$ is locally integrable: for some $Q_0\in\cD$,
	\begin{eqnarray*}
	\La b_\cS\Ra_{Q_0} &=& \frac{1}{|Q_0|} \bigg(\sum_{Q\in\cS, Q\subset Q_0} |Q| + \sum_{Q\in\cS, Q\supsetneq Q_0}|Q_0|\bigg)\\
		&=& \underbrace{\frac{1}{|Q_0|} \sum_{Q\in\cS, Q\subset Q_0} |Q|}_{\leq \Lambda}
		 + \underbrace{\#\{Q\in\cS: Q\supsetneq Q_0\}}_{<\infty \text{ because } \cS\in \Upsilon^\cD} < \infty.
	\end{eqnarray*}
Then, for some $Q_0\in\cD$:
	\begin{eqnarray*}
	(b_\cS-\La b_\cS\Ra_{Q_0})\unit_{Q_0} &=& \sum_{Q\in\cS, Q\subset Q_0} \unit_Q + \#\{Q\in\cS: Q\supsetneq Q_0\}\unit_{Q_0} - 
		\frac{\unit_{Q_0}}{|Q_0|}\sum_{Q\in\cS, Q\subset Q_0} |Q| - \#\{Q\in\cS: Q\supsetneq Q_0\}\unit_{Q_0}\\
		&=& \sum_{Q\in\cS, Q\subset Q_0} \unit_Q - \frac{\unit_{Q_0}}{|Q_0|}\sum_{Q\in\cS, Q\subset Q_0} |Q|.
	\end{eqnarray*}
In fact, we can reduce this further to
	\begin{equation}
	\label{E:bs-avg}
	(b_\cS-\La b_\cS\Ra_{Q_0})\unit_{Q_0} = \sum_{Q\in\cS, Q\subsetneq Q_0}\unit_Q - \frac{\unit_{Q_0}}{|Q_0|} \sum_{Q\in\cS, Q\subsetneq Q_0} |Q|,
	\end{equation}
which is clear if $Q_0 \notin \cS$, and if $Q_0\in\cS$ then $\unit_{Q_0}-\frac{\unit_{Q_0}}{|Q_0|}|Q_0|$ cancel.
A simple estimate then shows that
	$$
	\frac{1}{|Q_0|} \int_{Q_0} |b_\cS - \La b_\cS\Ra_{Q_0}|\,dx \leq 
		\frac{1}{|Q_0|} 2\sum_{Q\in\cS,Q\subsetneq Q_0}|Q| \leq 2\Lambda, \forall Q_0\in\cD,
	$$
so $b_\cS\in BMO^\cD(\bR^n)$. However, a more careful estimate is possible. We prove the following in Appendix \ref{A:my-sparse}.

\begin{theorem}
\label{T:SparseBMO}
Let $\cS\in\Upsilon^\cD(\bR^n)$ be a sparse collection with no infinitely increasing chains and Carleson constant $\Lambda$. Then the function
$b_\cS=\sum_{Q\in\cS}\unit_Q$ is in $BMO^\cD(\bR^n)$, with
	$$
	\|b_\cS\|_{BMO^\cD(\bR^n)} \leq \Lambda.
	$$
\end{theorem}

\vspace{0.1in}
\begin{center}
$\ast$
\end{center}
\vspace{0.1in}

This process works to yield a weighted BMO function as well: with any $\cS\in\Upsilon^\cD(\bR^n)$ and $w\in A_p^\cD$ we associate the function
	$$
	b_\cS^w := \sum_{Q\in\cS} \La w\Ra_Q \unit_Q.
	$$
As before, $\cS\in\Upsilon^\cD(\bR^n)$ ensures that $b_\cS^w$ is a.e. finite, locally integrable, and
	$$
	\unit_{Q_0}(b_\cS^w-\La b_\cS^w\Ra_{Q_0}) = \sum_{Q\in\cS, Q\subsetneq Q_0} \La w\Ra_Q \unit_Q - 
	\frac{\unit_{Q_0}}{|Q_0|} \sum_{Q\in\cS, Q\subsetneq Q_0} w(Q), \:\:\forall Q_0\in\cD.
	$$
By \eqref{E:w-sparse},
	$$
	\frac{1}{|Q_0|} \sum_{Q\in\cS, Q\subsetneq Q_0} w(Q) \leq [w]_{A_p} \Lambda^p \La w\Ra_{Q_0},
	$$
which then easily gives
	$$
	\frac{1}{w(Q_0)} \int_{Q_0}|b_\cS^w - \La b_\cS^w\Ra_{Q_0}|\,dx \leq 2[w]_{A_p} \Lambda^p,
	$$
so
	$$
	b_\cS^w \in BMO^\cD(w), \text{ with } \|b_\cS^w\|_{BMO^\cD(w)} \leq 2[w]_{A_p}\Lambda^p.
	$$


\subsection{Sparse Operators as Sums of Paraproducts and Martingale Transform.}
\label{sS:sparse=sum}
For ease of notation we work in $\bR$ below, but the obvious analog for $\bR^n$ follows easily in the same way. Consider
	$$
	\cA_\cS^w f := \sum_{I\in\cS} \La w\Ra_I \La f\Ra_I \unit_I,
	$$
where $\cS\in\Upsilon^\cD(\bR)$ and $w$ is an $A_p^\cD$ weight on $\bR$, $1<p<\infty$. 
A particularly interesting instance of $\cA_\cS^w$ occurs when $w=\nu \in A_2^\cD$, where $\nu:=\mu^{1/p}\lambda^{-1/p}$ for two weights $\mu, \lambda \in A_p^\cD$. We treat this operator in more detail in Section \ref{sS:BSpBd}.

Using the $b_\cS^w$ function associated with $\cS$ and $w$, we write
	\begin{equation}
	\label{E:sum}
	\cA_\cS^w f = \cA^w_\cS - b_\cS^w \cdot f + b_\cS^w \cdot f 
		= \cA_\cS^w f - b_\cS^w \cdot f + \big(\Pi_{b_\cS^w}f + \Pi_{b_\cS^w}^* f + \Pi_f b_\cS^w\big).
	\end{equation}
Now recall that
	$$
	\La b_\cS^w\Ra_{J_0} = (\tau^w_\cS)_{J_0} + \sum_{J\in\cS, J \supset J_0}\La w\Ra_J, \:\:\forall J_0 \in \cD,
	$$
where
	$$
	(\tau_\cS^w)_{J} := \frac{1}{|J|} \sum_{I\in\cS, I \subsetneq J} w(I), \:\:\forall J\in\cD,
	$$
a quantity always bounded if $w\in A_p^\cD$:
	$$
	(\tau_\cS^w)_{J} \leq [w]_{A_p} \Lambda^p \La w\Ra_J.
	$$
So:
	\begin{eqnarray*}
	\Pi_f b_\cS^w (x) &=& \sum_{J\in\cD} (f, h_J) \La b_\cS^w\Ra_J h_J(x) \\
		&=& \sum_{J\in\cD} (f, h_J) \big[ (\tau_\cS^w)_J + \sum_{K\in\cS, K\supset J} \La w\Ra_K \big] h_J(x)\\
		&=& \underbrace{(\tau_\cS^w)_J (f, h_J) h_J(x)}_{=: T_{\tau_\cS^w}f(x)} + \sum_{J\in\cD} (f, h_J) h_J(x) \bigg(\sum_{K\in\cS, K\supset J}\La w\Ra_K\bigg).
	\end{eqnarray*}
The second term can be further explored as
	\begin{eqnarray*}
	\sum_{J\in\cD} (f, h_J) h_J(x) \bigg(\sum_{K\in\cS, K\supset J}\La w\Ra_K\bigg) &=& \sum_{K\in\cS} \La w\Ra_K \bigg(
		\sum_{J\subset K} (f, h_J) h_J(x)\bigg)\\
		&=& \sum_{K\in\cS} \La w\Ra_K \bigg(f(x) - \La f\Ra_K\bigg)\unit_K(x)\\
		&=& f(x) \cdot \sum_{K\in\cS} \La w\Ra_K \unit_K(x) - \sum_{K\in\cS} \La w\Ra_K \La f\Ra_K \unit_K(x)\\
		&=& f(x)\cdot b_\cS^w(x) - \cA_\cS^w f(x).
	\end{eqnarray*}
	
Returning to \eqref{E:sum}:
	$$
	\cA_\cS^w f = \cA_\cS^w f - b_\cS^w \cdot f + (\Pi_{b_\cS^w}f + \Pi^*_{b_\cS^w}f) + T_{\tau_\cS^w}f + f\cdot b_\cS^w - \cA_\cS^w f,
	$$
so we have:

\begin{prop}
\label{P:SPM}
Any weighted sparse operator $\cA_\cS^w$, where $w\in A_p^\cD$ is a weight on $\bR$ and $\cS\in\Upsilon^\cD(\bR)$ is a sparse collection with no infinitely increasing chains, may be expressed as
	\begin{equation}
	\label{E:SPM}
	\cA_\cS^w f = \Pi_{b_\cS^w}f + \Pi^*_{b_\cS^w}f + T_{\tau_\cS^w}f,
	\end{equation}
where the first two terms are the paraproducts with symbol $b_\cS^w$, the sparse $BMO^\cD(w)$ function associated with $\cS$ and $w$, and the third term is
	$$
	T_{\tau_\cS^w} f(x) := \sum_{J\in\cD} (\tau^w_\cS)_J (f, h_J)h_J(x),\text{ where } 
	(\tau^w_\cS)_J := \frac{1}{|J|}\sum_{I\in\cS, I\subsetneq J} w(I) \leq [w]_{A_p}\Lambda^p \La w\Ra_J, \forall J\in \cD.
	$$
\end{prop}

\begin{remark}
In case $w\equiv 1$, we obtain the unweighted situation
	\begin{equation}
	\label{E:SPMu}
	\cA_\cS f = \Pi_{b_\cS}f + \Pi^*_{b_\cS}f + T_{\tau_\cS}f,
	\end{equation}
where $T_{\tau_\cS}$ is a martingale transform:
	$$
	T_{\tau_\cS} = \sum_{J\in\cD} (\tau_\cS)_J (f, h_J)h_J, \text{ where } (\tau_\cS)_J := \frac{1}{|J|}\sum_{I\in\cS, I\subsetneq J}|I| \leq \Lambda, \:\:\forall J\in\cD.
	$$
\end{remark}

\begin{remark}
In fact, \eqref{E:SPM} expresses sparse operators as Haar multipliers: recall that a Haar multiplier is an operator of the form
	$$
	T_\phi f(x) := \sum_{J\in\cD} \phi_J(x) (f, h_J) h_J(x),
	$$
where $\{\phi_J(x)\}_{J\in\cD}$ is a sequence of functions indexed by $\cD$. It is known that (see \cite{Blasco}):
	$$
	(\Pi_b + \Pi_b^*)f = \sum_J (b - \La b\Ra_J) (f, h_J)h_J.
	$$
So, from \eqref{E:SPM}:
	$$
	\cA_\cS^w f (x) = \bigg[ 
	\underbrace{(b_\cS^w(x) - \La b_\cS^w\Ra_J)\unit_J(x)  + (\tau_\cS^w)_J}_{\phi_J(x)}
	\bigg](f, h_J)h_J(x).
	$$
\end{remark}

\vspace{0.1in}
\begin{center}
$\ast$
\end{center}
\vspace{0.1in}

Look more closely now at \eqref{E:SPMu}: $\cA_\cS = \Pi_{b_\cS}+\Pi_{b_\cS}^* + T_{\tau_\cS}$. This gives an \textit{upper bound} for 
$\|\cA_\cS : L^p(w)\rightarrow L^p(w)\|$ in terms of the norms of paraproducts and martingale transform -- when usually it is the norms of sparse operators that are used as upper bounds:
	$$
	\|\cA_\cS f\|_{L^p(w)} \leq \|\Pi_{b_\cS}f\|_{L^p(w)} + \|\Pi^*_{b_\cS}f\|_{L^p(w)} + \|T_{\tau_\cS}f\|_{L^p(w)}.
	$$
Divide above by $\Lambda_{(\cS)} := \Lambda$, the Carleson constant of $\cS$, and recall that $\|b_\cS\|_{BMO^\cD}\leq\Lambda$, as well as 
$\|\tau_\cS\|_\infty \leq\Lambda$:
	\begin{eqnarray*}
	\frac{\|\cA_\cS f\|_{L^p(w)}}{\Lambda} &\leq& \frac{\|\Pi_{b_\cS}f\|_{L^p(w)}}{\Lambda} + \frac{\|\Pi^*_{b_\cS}f\|_{L^p(w)}}{\Lambda} +
				\frac{\|T_{\tau_\cS}f\|_{L^p(w)}}{\Lambda}\\
	&\leq& \frac{\|\Pi_{b_\cS}f\|_{L^p(w)}}{\|b_\cS\|_{BMO^\cD}} + \frac{\|\Pi^*_{b_\cS}f\|_{L^p(w)}}{\|b_\cS\|_{BMO^\cD}} +
				\frac{\|T_{\tau_\cS}f\|_{L^p(w)}}{\|\tau_\cS\|_\infty},
	\end{eqnarray*}
from which we can deduce that, for all $\Lambda>1$:
	\begin{eqnarray*}
	\sup_{\substack{\cS\in\Upsilon^\cD(\bR) \\ \Lambda_{(\cS)}=\Lambda}} \frac{\|\cA_\cS : L^p(w)\rightarrow L^p(w)\|}{\Lambda} &\leq&
		\sup_{\substack{b\in BMO^\cD \\ \|b\|_{BMO^\cD}\neq 0}} \frac{\|\Pi_b : L^p(w)\rightarrow L^p(w)\|}{\|b\|_{BMO^\cD}}
		+ \sup_{\substack{b\in BMO^\cD \\ \|b\|_{BMO^\cD}\neq 0}} \frac{\|\Pi_b^* : L^p(w)\rightarrow L^p(w)\|}{\|b\|_{BMO^\cD}}\\
		&&+ \sup_{\substack{\tau\in\ell^\infty\\ \|\tau\|_\infty\neq 0}} \frac{\|T_\tau : L^p(w)\rightarrow L^p(w)\|}{\|\tau\|_\infty}.
	\end{eqnarray*}
Given the well-known domination results \cite{Lacey} for the martingale transform and paraproducts:
	\begin{eqnarray*}
	\sup_{\substack{\cS\in\Upsilon^\cD(\bR) \\ \Lambda_{(\cS)}=\Lambda}} \|\cA_\cS : L^p(w)\rightarrow L^p(w)\| &\simeq_{\Lambda,p}&
		\sup_{\substack{b\in BMO^\cD \\ \|b\|_{BMO^\cD}\neq 0}} \frac{\|\Pi_b : L^p(w)\rightarrow L^p(w)\|}{\|b\|_{BMO^\cD}}
		+ \sup_{\substack{b\in BMO^\cD \\ \|b\|_{BMO^\cD}\neq 0}} \frac{\|\Pi_b^* : L^p(w)\rightarrow L^p(w)\|}{\|b\|_{BMO^\cD}}\\
		&&+ \sup_{\substack{\tau\in\ell^\infty\\ \|\tau\|_\infty\neq 0}} \frac{\|T_\tau : L^p(w)\rightarrow L^p(w)\|}{\|\tau\|_\infty}.
	\end{eqnarray*}
	
\begin{remark}
It would be interesting if the martingale and paraproducts can be ``separated'' somehow, and to obtain independently that paraproducts and martingale transforms have the same dependency on $[w]_{A_p}$ by showing they are both equivalent to $\|\cA_\cS\|$. However, we can show that the norms of $\cA_\cS$ are equivalent to norms of certain compositions of paraproducts. We do this next.
\end{remark}

\subsection{Sparse Operators and Compositions of Paraproducts} 
\label{sS:ParaComp}
Consider the composition 
$$
\Pi^*_a\Pi_bf = \sum_{Q\in\cD}(a, h_Q)(b, h_Q)\La f\Ra_Q \frac{\unit_Q}{|Q|}.
$$
We show in Appendix \ref{A:ParaComp}, using a bilinear form argument, that:
	\begin{theorem}
	\label{T:ParaComp}
	There is a dimensional constant $C(n)$ such that for all Bloom weights $\mu,\lambda\in A_p$ ($1<p<\infty$), $\nu:=\mu^{1/p}\lambda^{-1/p}$ on $\bR^n$,
	BMO functions $a\in BMO^\cD(\bR^n)$, $b\in BMO^\cD(\nu)$, and $\Lambda>1$:
	$$
	\|\Pi_a^*\Pi_b :L^p(\mu)\rightarrow L^p(\lambda)\| \leq C(n) \|a\|_{BMO^\cD}\|b\|_{BMO^\cD(\nu)} 
	\sup_{\substack{\cS\in\Upsilon^\cD(\bR^n)\\ \Lambda_{(\cS)}=\Lambda}} \left(\frac{\Lambda}{\Lambda-1}\right)^3 
	\|\cA_\cS^\nu:L^p(\mu) \rightarrow L^p(\lambda)\|.
	$$
	\end{theorem}
Some immediate observations about this result:
	\begin{itemize}
	\item From Theorem \ref{T:Bspbd}:
		$$
		\|\Pi_a^*\Pi_b : L^p(\mu)\rightarrow L^p(\lambda)\|\lesssim \|a\|_{BMO^\cD}\|b\|_{BMO^\cD(\nu)}[\mu]_{A_p}^{\frac{1}{p-1}}[\lambda]_{A_p}.
		$$
	\item Take $\mu=\lambda=w$, for some $w\in A_p$. Then $\nu=1$ and we obtain in the one-weight situation:
		\begin{equation}
		\label{E:ParaComp-w}
		\|\Pi_a^*\Pi_b:L^p(w)\rightarrow L^p(w)\|\lesssim \|a\|_{BMO^\cD}\|b\|_{BMO^\cD}[w]_{A_p}^{\max\left(1, \frac{1}{p-1}\right)}.
		\end{equation}
	\item It is easy to see that $\Pi_a^*\Pi_b = \Pi_b^*\Pi_a$, so the same result holds for $\Pi_b^*\Pi_a$, with $b\in BMO^\cD(\nu)$, $a\in BMO^\cD$.
	\end{itemize}

\vspace{0.1in}
\begin{center}
$\ast$
\end{center}
\vspace{0.1in}

Let $\cS\in\Upsilon^\cD(\bR^n)$. We associated with $\cS$ the BMO function $b_\cS=\sum_{Q\in\cS}\unit_Q$. There is another, even more obvious BMO function we can associate with $\cS$:
	$$
	\widetilde{b}_\cS :=\sum_{Q\in\cS}\sqrt{|Q|}h_Q = \sum_{\substack{Q\in\cS\\ \epsilon\neq 1}}\sqrt{|Q|}h_Q^\epsilon.
	$$
For any $Q_0\in\cD$:
	$$
	\frac{1}{|Q_0|}\int_{Q_0}|\widetilde{b}_\cS - \La\widetilde{b}_\cS\Ra_{Q_0}|^2\,dx = 
	\frac{1}{|Q_0|} \sum_{\substack{Q\subset Q_0, Q\in\cS \\ \epsilon\neq 1}}|Q|
	= (2^n-1)\frac{1}{|Q_0|}\sum_{Q\subset Q_0, Q\in\cS}|Q| \leq (2^n-1)\Lambda,
	$$
so
	$$
	\|\widetilde{b}_\cS\|_{BMO^\cD} \leq \sqrt{(2^n-1)\Lambda}.
	$$
Moreover, 
	$$
	\Pi^*_{\widetilde{b}_\cS}\Pi_{\widetilde{b}_\cS}f = \sum_{\substack{Q\in\cD\\ \epsilon\neq 1}} 
	(\widetilde{b}_\cS, h_Q^\epsilon)^2\La f\Ra_Q\frac{\unit_Q}{|Q|}
	= \sum_{\substack{Q\in\cS\\ \epsilon\neq 1}}|Q|\La f\Ra_Q\frac{\unit_Q}{|Q|}
	= (2^n-1)\sum_{Q\in\cS}\La f\Ra_Q\unit_Q,
	$$
so we may express the sparse operator $\cA_\cS$ as
	$$
	\cA_\cS  = \frac{1}{2^n-1} \Pi^*_{\widetilde{b}_\cS}\Pi_{\widetilde{b}_\cS}.
	$$
Then
	$$
	\frac{\|\cA_\cS f\|_{L^p(w)}}{\Lambda} = \frac{1}{2^n-1}\frac{\|\Pi^*_{\widetilde{b}_\cS}\Pi_{\widetilde{b}_\cS}f\|_{L^p(w)}}{\Lambda}
	\leq \frac{\|\Pi^*_{\widetilde{b}_\cS}\Pi_{\widetilde{b}_\cS}f\|_{L^p(w)}}{\|\widetilde{b}_\cS\|_{BMO^\cD}^2}
	\leq \sup_{a,b\in BMO^\cD} \frac{\|\Pi^*_a\Pi_b f\|_{L^p(w)}}{\|a\|_{BMO^\cD} \|b\|_{BMO^\cD}},
	$$
which means that for all $\Lambda>1$:
	$$
	\sup_{\substack{\cS\in\Upsilon^\cD(\bR^n) \\ \Lambda_{(\cS)}=\Lambda}} \frac{\|\cA_\cS:L^p(w)\rightarrow L^p(w)\|}{\Lambda}
	\leq \sup_{a,b\in BMO^\cD} \frac{\|\Pi^*_a\Pi_b :L^p(w)\rightarrow L^p(w)\|}{\|a\|_{BMO^\cD}\|b\|_{BMO^\cD}}.
	$$
Combined with \eqref{E:ParaComp-w}, we have
	$$
	\sup_{\cS\in\Upsilon^\cD(\bR^n)} \|\cA_\cS:L^p(w)\rightarrow L^p(w)\| \simeq_{\Lambda, p, n} \sup_{a,b\in BMO^\cD(\bR^n)} 
	\frac{\|\Pi^*_a\Pi_b :L^p(w)\rightarrow L^p(w)\|}{\|a\|_{BMO^\cD}\|b\|_{BMO^\cD}}.
	$$

\subsection{The Bloom Sparse Operator $\cA_\cS^\nu$} 
\label{sS:BSpBd}
Consider
	$$
	\cA_\cS^\nu f = \sum_{Q\in\cS}\La\nu\Ra_Q\La f\Ra_Q\unit_Q,
	$$
for a sparse collection $\cS\subset\cD(\bR^n)$, where $\mu,\lambda\in A_p$ ($1<p<\infty$) and $\nu:=\mu^{1/p}\lambda^{-1/p}$ are Bloom weights.
In looking to bound this operator $L^p(\mu)\rightarrow L^p(\lambda)$, the first obvious route is to appeal to the known one-weight bounds for the usual, unweighted sparse operator $\cA_\cS f = \sum_{Q\in\cS}\La f\Ra_Q\unit_Q$. We want something like $\|\cA_\cS^\nu f\|_{L^p(\lambda)} \leq C \|f\|_{L^p(\mu)}$, and we use duality to express
	$$
	\|\cA_\cS^\nu f\|_{L^p(\lambda)} = \sup_{\substack{g\in L^{p'}(\lambda') \\ \|g\|_{L^{p'}(\lambda')} \leq 1}} |(\cA_\cS^\nu f, g)|.
	$$
So we look for a bound of the type $|(\cA_\cS^\nu f, g)| \leq C \|f\|_{L^p(\mu)} \|g\|_{L^{p'}(\lambda')}$.

\begin{eqnarray*}
|(\cA_\cS^\nu f, g)| &=& \bigg| \sum_{Q\in\cS} \La \nu\Ra_Q \La f\Ra_Q \La g\Ra_Q |Q| \bigg| 
\leq \sum_{Q\in\cS}\La|f|\Ra_Q\La|g|\Ra_Q \nu(Q)
\leq \int(\sum_{Q\in\cS} \La|f|\Ra_Q\La|g|\Ra_Q\unit_Q)\,d\nu\\
&\leq& \int (\cA_\cS|f|)(\cA_\cS|g|)\mu^{1/p}\lambda^{-1/p}\,dx
\leq \|\cA_\cS |f|\|_{L^p(\mu)}\|\cA_\cS|g|\|_{L^{p'}(\lambda')}\\
&\leq& \|\cA_\cS : L^p(\mu)\rightarrow L^p(\mu)\| \cdot \|\cA_\cS : L^{p'}(\lambda')\rightarrow L^{p'}(\lambda')\| \cdot \|f\|_{L^p(\mu)}\|g\|_{L^{p'}(\lambda')}.
\end{eqnarray*}

This yields the same dependency on the $A_p$ characteristics of $\mu, \lambda$ as obtained in \cite{Lerner} for commutators:
	$$
	\|\cA_\cS^\nu : L^p(\mu)\rightarrow L^p(\lambda)\| \lesssim \big([\mu]_{A_p} [\lambda]_{A_p}\big)^{\max\left(1, \frac{1}{p-1}\right)}
	$$
We give another proof, inspired by the beautiful proof in \cite{DCU} of the $A_2$ conjecture for usual unweighted sparse operators, which yields a smaller bound.

\begin{theorem}
\label{T:Bspbd}
Let $\cS\subset\cD$ be a sparse collection of dyadic cubes, $\mu, \lambda\in A_p^\cD$, $1<p<\infty$ be two $A_p$ weights on $\bR^n$, and $\nu:=\mu^{1/p}\lambda^{-1/p}$. Then the Bloom sparse operator
	$$
	\cA_\cS^\nu f := \sum_{Q\in\cS}\La\nu\Ra_Q\La f\Ra_Q\unit_Q
	$$
is bounded $L^p(\mu)\rightarrow L^p(\lambda)$ with
	\begin{equation}
	\label{E:Bspbd}
	\|\cA_\cS^\nu : L^p(\mu) \rightarrow L^p(\lambda)\| \leq \Lambda^{p'+p-2} (p')^2 [\mu']_{A_{p'}} [\lambda]_{A_p}
	= \Lambda^{p'+p-2} (pp') [\mu]_{A_{p}}^{\frac{1}{p-1}} [\lambda]_{A_p},
	\end{equation}
where $\Lambda$ is the Carleson constant of $\cS$.
\end{theorem}

\begin{proof}
In looking for a bound of the type $\|\cA_\cS^\nu f\|_{L^p(\lambda)}\leq C\|f\|_{L^p(\mu)}$, consider instead $\varphi:=f\mu'$: then 
$\|\varphi\|_{L^p(\mu)} = \|f\|_{L^p(\mu')}$, so we look instead for a bound of the type
$\|\cA_\cS^\nu(f\mu')\|_{L^p(\lambda)}\leq C\|f\|_{L^{p}(\mu')}$. Using the standard $L^p(\lambda) - L^{p'}(\lambda)$ duality with 
$(\cdot,\cdot)_{d\lambda}$ inner product, we write
	$$
	\|\cA_\cS^\nu(f\mu')\|_{L^p(\lambda)} = \sup_{\substack{g\in L^{p'}(\lambda)\\ \|g\|_{L^{p'}(\lambda)}\leq 1}} |(\cA_\cS^\nu (f\mu'), g\lambda)|,
	$$
meaning we finally look for a bound of the type
	$$
	|(\cA_\cS^\nu(f\mu'), g\lambda)| \leq C \|f\|_{L^p(\mu')}\|g\|_{L^{p'}(\lambda)}.
	$$
	
As in \cite{DCU}, we make use of the weighted dyadic maximal function:
	$$
	M_u^\cD f(x) := \sup_{Q\in\cD} \bE_Q^u |f| \unit_Q(x),
	$$
and its property of being $L^q(u)$-bounded with a constant independent of $u$:
	\begin{theorem}
	For any locally finite Borel measure $u$ on $\bR^n$ and any $q\in(1,\infty)$:
		\begin{equation}
		\label{E:wMax}
		\|M_u^\cD : L^q(u)\rightarrow L^q(u) \| \leq q'.
		\end{equation}
	\end{theorem}
\noindent See, for example, \cite{HytonenLectures} for a proof of this fact.

Now:
	$$
	|(\cA_\cS^\nu (f\mu'), g\lambda)| = |\sum_{Q\in\cS} \La\nu\Ra_Q \La f\mu'\Ra_Q \La g\lambda\Ra_Q|Q||
	\leq \sum_{Q\in\cS} \La|f|\mu'\Ra_Q \La|g|\lambda\Ra_Q \La\nu\Ra_Q|Q|.
	$$
We express the averages involving $f$ and $g$ as weighted averages:
	$$
	\sum_{Q\in\cS} \La|f|\mu'\Ra_Q \La|g|\lambda\Ra_Q \La\nu\Ra_Q|Q| =
	\sum_{Q\in\cS} \bigg(\bE_Q^{\mu'}|f|\bigg) \La\mu'\Ra_Q \bigg(\bE_Q^\lambda|g|\bigg) \La\lambda\Ra_Q \La\nu\Ra_Q|Q|.
	$$
Apply the fact that $\La\nu\Ra_Q\leq \La\mu\Ra_Q^{1/p}\La\lambda'\Ra_Q^{1/p'}$ (an easy consequence of H\"{o}lder's inequality), and the fact that for any $A_p$ weight $w$, we have 
	$$[w]^{1/p}_{A_p} =\sup_Q \La w\Ra_Q^{1/p}\La w'\Ra_Q^{1/p'},$$
to go further:
	\begin{eqnarray*}
	|(\cA_\cS^\nu (f\mu'), g\lambda)| &\leq& 
		\sum_{Q\in\cS} \bigg(\bE_Q^{\mu'}|f|\bigg) \bigg(\bE_Q^\lambda|g|\bigg) \La\mu'\Ra_Q  \La\lambda\Ra_Q \La\mu\Ra_Q^{1/p}\La\lambda'\Ra_Q^{1/p'}|Q|\\
	&\leq& [\mu]_{A_p}^{1/p}[\lambda]_{A_p}^{1/p} \sum_{Q\in\cS} 
		\bigg(\bE_Q^{\mu'}|f|\bigg) \bigg(\bE_Q^\lambda|g|\bigg) \La\mu'\Ra_Q^{1/p}\La\lambda\Ra_Q^{1/p'}|Q|\\
	&=& [\mu]_{A_p}^{1/p}[\lambda]_{A_p}^{1/p} \sum_{Q\in\cS} 
		\bigg(\bE_Q^{\mu'}|f|\bigg) \bigg(\bE_Q^\lambda|g|\bigg) \mu'(Q)^{1/p}\lambda(Q)^{1/p'}\\
	&\leq& [\mu]_{A_p}^{1/p}[\lambda]_{A_p}^{1/p} 
		\left(\sum_{Q\in\cS} \bigg(\bE_Q^{\mu'}|f|\bigg)^p \mu'(Q)\right)^{1/p}
		\left(\sum_{Q\in\cS} \bigg(\bE_Q^\lambda|g|\bigg)^{p'}\lambda(Q)\right)^{1/p'}
	\end{eqnarray*}
Now apply \eqref{E:w-sparse}:
	$$
	\mu'(Q)\leq[\mu']_{A_{p'}}\lambda^{p'}\mu'(E_Q) = [\mu]_{A_p}^{p'-1}\Lambda^{p'}\mu'(E_Q)
	\text{ and } \lambda(Q)\leq[\lambda]_{A_p}\Lambda^p\lambda(E_Q),
	$$
so we may later use disjointness of the sets $\{E_Q\}_{Q\in\cS}$.

\begin{minipage}[t]{0.475\textwidth}
\begin{eqnarray*}
&&\left(\sum_{Q\in\cS} \bigg(\bE_Q^{\mu'}|f|\bigg)^p \mu'(Q)\right)^{1/p} \\
&\leq& [\mu]_{A_p}^{\frac{p'-1}{p}}\Lambda^{p'/p}\left(\sum_{Q\in\cS} \bigg(\bE_Q^{\mu'}|f|\bigg)^p \mu'(E_Q)\right)^{1/p}\\
&=& [\mu]_{A_p}^{\frac{p'-1}{p}}\Lambda^{p'/p} \left(\sum_{Q\in\cS} \int_{E_Q} \bigg(\bE_Q^{\mu'}|f|\bigg)^p \,d\mu' \right)^{1/p}\\
&\leq& [\mu]_{A_p}^{\frac{p'-1}{p}}\Lambda^{p'/p} \left(\sum_{Q\in\cS} \int_{E_Q} \bigg(M_{\mu'}^\cD f\bigg)^p \,d\mu'\right)^{1/p}\\
&\leq& [\mu]_{A_p}^{\frac{p'-1}{p}}\Lambda^{p'/p} \left(\int_{\bR^n}\bigg(M_{\mu'}^\cD f\bigg)^p\,d\mu'\right)^{1/p}\\
&=& [\mu]_{A_p}^{\frac{p'-1}{p}}\Lambda^{p'/p} \|M_{\mu'}^\cD f\|_{L^p(\mu')}\\
&\leq& [\mu]_{A_p}^{\frac{p'-1}{p}}\Lambda^{p'/p} p' \|f\|_{L^p(\mu')}
\end{eqnarray*}
\end{minipage}
\hfill
\noindent\begin{minipage}[t]{0.475\textwidth}
\begin{eqnarray*}
&& \left(\sum_{Q\in\cS} \bigg(\bE_Q^\lambda|g|\bigg)^{p'}\lambda(Q)\right)^{1/p'}\\
&\leq& [\lambda_{A_p}^{1/p'}\Lambda^{p/p'} \left(\sum_{Q\in\cS} \bigg(\bE_Q^\lambda|g|\bigg)^{p'} \lambda(E_Q)\right)^{1/p'}\\
&\leq& [\lambda_{A_p}^{1/p'}\Lambda^{p/p'} \left(\sum_{Q\in\cD} \int_{E_Q} \bigg(M_\lambda^\cD g\bigg)^{p'}\,d\lambda\right)^{1/p'}\\
&\leq& [\lambda_{A_p}^{1/p'}\Lambda^{p/p'} \|M_\lambda^\cD g\|_{L^{p'}(\lambda)}\\
&\leq& [\lambda_{A_p}^{1/p'}\Lambda^{p/p'} p \|g\|_{L^{p'}(\lambda)}.
\end{eqnarray*}
\end{minipage}

Putting these estimates together:

\begin{eqnarray*}
|(\cA_\cS^\nu (f\mu'), g\lambda)| &\leq&  [\mu]_{A_p}^{1/p}[\lambda]_{A_p}^{1/p} 
	[\mu]_{A_p}^{\frac{p'-1}{p}}\Lambda^{p'/p} p' \|f\|_{L^p(\mu')}
	[\lambda_{A_p}^{1/p'}\Lambda^{p/p'} p \|g\|_{L^{p'}(\lambda)}\\
&=& [\mu]_{A_p}^{p'/p} [\lambda]_{A_p} \Lambda^{p'/p + p/p'} pp' \|f\|_{L^p(\mu')} \|g\|_{L^{p'}(\lambda)}\\
&=& [\mu']_{A_{p'}}[\lambda]_{A_p} \Lambda^{p+p'-2}pp' \|f\|_{L^p(\mu')}\|g\|_{L^{p'}(\lambda)},
\end{eqnarray*}
which proves the theorem.
\end{proof}


\section{Paraproducts and Bloom BMO}
\label{S:ParaDom}

We show the following pointwise domination result, inspired by ideas in \cite{Lacey} on pointwise domination of the martingale transform.

\begin{theorem}
\label{T:ParaDom}
There is a dimensional constant $C(n)$ such that: for every $\Lambda>1$, weight $w$ on $\bR^n$, $b\in BMO^\cD(w)$, fixed dyadic cube $Q_0\in\cD$ and $f \in L^1(Q_0)$, there is a $\Lambda$-Carleson sparse collection $\cS(Q_0) \subset \cD(Q_0)$ (depending on $b, w, f$) such that:
	$$
	\forall x\in Q_0: \:\: |\Pi_{b, Q_0}f(x)| \leq C(n) \bigg(\frac{\Lambda}{\Lambda-1}\bigg)^2 \|b\|_{BMO^\cD(w)} 
	\cA^w_{\cS(Q_0)}|f|(x).
	$$
The same holds for the other paraproducts $\Pi_b^*$ and $\Gamma_b$.
\end{theorem}

Assuming this, return to the Bloom situation for a moment and say $b\in BMO^\cD(\nu)$ has finite Haar expansion. Then there are at most $2^n$ disjoint dyadic cubes $\{Q_k\}_{1\leq k\leq 2^n} \subset \cD$ such that
$b=\sum_K \sum_{Q\subset Q_k}(b, h_Q)h_Q$, and then $\Pi_bf = \sum_k \Pi_{b, Q_k}f$. So, assuming Theorem \ref{T:ParaDom}, there are $\Lambda$-Carleson sparse collections $\cS(Q_k)\subset \cD(Q_k)$ such that
	\begin{eqnarray*}
	|\Pi_bf(x)| &\leq& \sum_k |\Pi_{b, Q_k}f(x)|\\
	&\leq& C(n)\left(\frac{\Lambda}{\Lambda-1}\right)^2 \|b\|_{BMO^\cD(\nu)} \sum_k \cA^\nu_{\cS(Q_k)}|f|(x)\\
	&=& C(n)\left(\frac{\Lambda}{\Lambda-1}\right)^2 \|b\|_{BMO^\cD(\nu)} \cA^\nu_{\cS}|f|(x),
	\end{eqnarray*}
where $\cS$ is a sparse collection with Carleson constant $\Lambda$ and no infinitely increasing chains:
	$$
	\cS=\cup_{k}\cS(Q_k) \in \Upsilon^\cD(\bR^n) \text{ with } \Lambda_{(\cS)}=\Lambda.
	$$
So
	$$
	\|\Pi_b :L^p(\mu) \rightarrow L^p(\lambda)\| \leq C(n) \|b\|_{BMO^\cD(\nu)}  \sup_{\substack{\cS\in\Upsilon^\cD \\ \Lambda_{(\cS)}=\Lambda}}
		\left(\frac{\Lambda}{\Lambda-1}\right)^2 \|\cA^\nu_{\cS} : L^p(\mu)\rightarrow L^p(\lambda)\|
	$$
holds for all $b\in BMO^\cD(\nu)$ with finite Haar expansion -- and thus for all $b$.

\begin{corollary}
Given Bloom weights $\mu,\lambda\in A_p^\cD$, $\nu=\mu^{1/p}\lambda^{-1/p}$, for all $b\in BMO^\cD(\nu)$:
	$$
	\|\Pi_b :L^p(\mu) \rightarrow L^p(\lambda)\| \leq C(n) \|b\|_{BMO^\cD(\nu)}  \sup_{\substack{\cS\in\Upsilon^\cD \\ \Lambda_{(\cS)}=\Lambda}}
		\left(\frac{\Lambda}{\Lambda-1}\right)^2 \|\cA^\nu_{\cS} : L^p(\mu)\rightarrow L^p(\lambda)\|.
	$$
The same holds for the other paraproducts $\Pi^*_b$ and $\Gamma_b$.
\end{corollary}

In light of the bound for $\cA_\cS^\nu$ in Theorem \ref{T:Bspbd}, pick some value for $\Lambda$, say $\Lambda=2$, and we have:

\begin{theorem}
\label{T:ParaBd}
Given Bloom weights $\mu,\lambda\in A_p^\cD$, $\nu=\mu^{1/p}\lambda^{-1/p}$, for all $b\in BMO^\cD(\nu)$:
	$$
	\|\Pi_b :L^p(\mu) \rightarrow L^p(\lambda)\| \leq C(n,p) \|b\|_{BMO^\cD(\nu)} [\mu]_{A_p}^{\frac{1}{p-1}}[\lambda]_{A_p}.
	$$
The same holds for the other paraproducts $\Pi^*_b$ and $\Gamma_b$.
\end{theorem}

\begin{remark}
The result actually follows immediately for $\Pi_b^*$, since
	$$
	\|\Pi_b :L^p(\mu) \rightarrow L^p(\lambda)\| = \|\Pi_b^* :L^{p'}(\lambda') \rightarrow L^{p'}(\mu')\|
	$$
and
	$$
	\nu'=(\lambda')^{1/p'}(\mu')^{-1/p'}=(\lambda^{-p'/p})^{1/p'}(\mu^{-p'/p})^{-1/p'}=\nu.
	$$
\end{remark}

\begin{remark}
As discussed in the introduction, we do not know if this bound is sharp -- but we can show that one particular instance of this inequality is sharp -- namely when $\mu=w$ and $\lambda = w^{-1}$ for some $A_2^\cD$ weight $w$, in which case the ``intermediary'' Bloom weight is also $\nu=w$:
\begin{equation}
\label{E:para2}
\|\Pi_b : L^2(w)\rightarrow L^2(w^{-1})\|\lesssim \|b\|_{BMO(w)^\cD}[w]_{A_2}^2
\end{equation}
\end{remark}


\subsection{Proof that the quadratic bound $[w]_{A_2}^2$ in \eqref{E:para2} is sharp (via the one-weight linear $A_2$ bound for the dyadic square function).}
\label{Ss:sharp}
The starting point is a simple observation: Given a weight $w$ on $\bR^n$, the \textit{weight itself belongs to $BMO(w)$}, with 
	$$\|w\|_{BMO(w)} \leq 2.$$
To see this, if $Q$ is a cube:
	$$
	\frac{1}{w(Q)} \int_Q|w(x)-\La w\Ra_Q|\,dx \leq \frac{1}{w(Q)} (w(Q)+w(Q))=2.
	$$
So we may look at the paraproducts with symbol $w$: in $\bR$ these are
	\begin{eqnarray*}
	\Pi_w f &=& \sum_{I\in\cD}(w, h_I) \La f\Ra)_I h_I\\
	\Pi^*_w f &=& \sum_{I\in\cD} (w, h_I) (f, h_I) \frac{\unit_I}{|I|}.
	\end{eqnarray*}
If $w\in A_2^\cD$, these are bounded
	$$
	\|\Pi_w : L^2(w)\rightarrow L^2(w^{-1})\| = \|\Pi_w^* : L^2(w)\rightarrow L^2(w^{-1})\| \lesssim \|w\|_{BMO^\cD(w)}[w]_{A_2}^2 = 2[w]^2_{A_2}.
	$$
	
Recall the decomposition
	$$
	fw = \Pi_w f + \Pi_w^* f + \Pi_f w
	$$
and note that the map $f \mapsto fw$ is an isometry $L^2(w)\rightarrow L^2(w^{-1})$. So
	$$
	\Pi_f w = \sum_{I\in\cD} (f, h_I)\La w\Ra_I h_I
	$$
is bounded $L^2(w)\rightarrow L^2(w^{-1})$:
	$$
	\|\Pi_f w\|_{L^2(w^{-1})} \leq \left(1 + 2 \|\Pi_w :L^2(w)\rightarrow L^2(w^{-1})\| \right)\|f\|_{L^2(w)}.
	$$
Now look at the $L^2(w)$-norm of the dyadic square function $S_\cD f := (\sum_{I}(f, h_I)^2 \frac{\unit_I}{|I|})^{1/2}$:
	$$
	\|S_\cD f\|^2_{L^2(w)} = \sum_{I\in\cD} (f, h_I)^2 \La w\Ra_I
	= \left(f, \sum_{I\in\cD} (f, h_I)\La w\Ra_I h_I\right)
	= (f, \Pi_f w) \leq \|\Pi_f w\|_{L^2(w^{-1})}\|f\|_{L^2(w)},
	$$
so
	$$
	\|S_\cD f\|^2_{L^2(w)} \leq \left(1+ 2\|\Pi_w :L^2(w)\rightarrow L^2(w^{-1})\|\right) \|f\|^2_{L^2(w)}.
	$$
Since
	\begin{equation}
	\label{E:parstar}
	\frac{1}{2}\leq\frac{\|\Pi_w :L^2(w)\rightarrow L^2(w^{-1})\|}{\|w\|_{BMO(w)}}
	\end{equation}
(we will show this in a moment) and
	$$
	\frac{1}{2} \leq \frac{1}{\|w\|_{BMO(w)}},
	$$
we have further that
	$$
	\|S_\cD f\|^2_{L^2(w)} \leq \|f\|^2_{L^2(w)} \left( 2\frac{\|\Pi_w :L^2(w)\rightarrow L^2(w^{-1})\|}{\|w\|_{BMO(w)}} + 4\frac{\|\Pi_w :L^2(w)\rightarrow L^2(w^{-1})\|}{\|w\|_{BMO(w)}}	\right),
	$$
which yields
	$$
	\frac{\|S_\cD f\|_{L^2(w^{-1})}}{\|f\|_{L^2(w)}} \leq \sqrt{6} \left(\frac{\|\Pi_w :L^2(w)\rightarrow L^2(w^{-1})\|}{\|w\|_{BMO(w)}}\right)^{1/2}
	\leq \sqrt{6} \sup_{b\in BMO^\cD(w)} \left(\frac{\|\Pi_b :L^2(w)\rightarrow L^2(w^{-1})\|}{\|b\|_{BMO(w)}}\right)^{1/2}.
	$$
Finally, the fact that 
	$$
	\sup_{b\in BMO^\cD(w)} \left(\frac{\|\Pi_b :L^2(w)\rightarrow L^2(w^{-1})\|}{\|b\|_{BMO(w)}}\right) \geq \frac{1}{6} \|S_\cD : L^2(w)\rightarrow L^2(w)\|^2 \simeq [w]_{A_2}
	$$
shows that any smaller bound in \eqref{E:para2} would imply a bound for $\|S_\cD : L^2(w)\rightarrow L^2(w)\|$ smaller than $[w]_{A_2}$, which is well-known to be false.

Going back to \eqref{E:parstar}, it is easy to show that
	$$
	\unit_Q(b-\La b\Ra_Q) = \unit_Q (\Pi_b \unit_Q - \Pi^*_b\unit_Q), \:\:\forall Q\in\cD.
	$$
Then
	\begin{eqnarray*}
	\frac{1}{w(Q)}\int_Q|b - \La b\Ra_Q|\,dx &=& \frac{1}{w(Q)}\int_Q |\Pi_b\unit_Q - \Pi_b^*\unit_Q|\,dx\\
		&\leq& \frac{1}{w(Q)}\left[\bigg(\int_Q |\Pi_b \unit_Q|^2\,dw^{-1}\bigg)^{1/2}w(Q)^{1/2}
			+ \bigg(\int_Q |\Pi_b^* \unit_Q|^2\,dw^{-1}\bigg)^{1/2}w(Q)^{1/2} \right]\\
	&\leq& \frac{1}{w(Q)^{1/2}} 2\|\Pi_b :L^2(w)\rightarrow L^2(w^{-1})\| \:\: \|\unit_Q\|_{L^2(w)},
	\end{eqnarray*}
which gives us
	$$
	\|b\|_{BMO^\cD(w)} \leq 2 \|\Pi_b :L^2(w)\rightarrow L^2(w^{-1})\|, \forall b\in BMO^\cD(w).
	$$

\vspace{0.1in}
\begin{center}
$\ast$
\end{center}
\vspace{0.2in}

Now we proceed with the proof of Theorem \ref{T:ParaDom}, focusing on $\Pi_b$, with the other paraproducts following similarly.


\subsection{Maximal Truncation of Paraproducts.} Let $b\in BMO_{\cD}(\bR^n)$. Define the maximal truncation of the paraproduct $\Pi_b$:
	$$
	\overset{\vartriangleright}{\Pi_b}f(x) := \sup_{P\in\cD} \left| \sum_{Q \supsetneq P} (b, h_Q)\La f\Ra_Q h_Q(x)\right|.
	$$
We will need the following result, which may be found in Lemma 2.10 of \cite{PereyraLectures}. 
\begin{prop}
\label{P:Pereyra}
Suppose $T:L^2(\bR^n) \rightarrow L^2(\bR^n)$ is a bounded linear or sublinear operator. If $T$ satisfies
	$$
	\supp(Th_Q) \subset Q, \forall Q\in\cD,
	$$
then $T$ is of weak $(1, 1)$ type, with
	$$
	|\{x: |Tf(x)|>\alpha\}| \leq C_n B \frac{1}{\alpha}\|f\|_1,
	$$
where $C_n$ is a dimensional constant and $B:=\|T\|_{L^2\rightarrow L^2}$.
\end{prop}

Now we prove some properties of $\overset{\vartriangleright}{\Pi_b}$.
\begin{prop}
\label{P:MaxTrunc}
The maximal truncation defined above satisfies the following:
	\begin{enumerate}[i.]
	\item $\overset{\vartriangleright}{\Pi_b}$ dominates $\Pi_b$:
		$$
		|\Pi_b f(x)| \leq \overset{\vartriangleright}{\Pi_b}f(x), \forall x\in\bR^n.
		$$
		
	\item $\overset{\vartriangleright}{\Pi_b}$ is dominated by $M^\cD\Pi_b$:
		$$
		\overset{\vartriangleright}{\Pi_b}f(x) \leq M^\cD (\Pi_b f)(x), \forall x\in \bR^n.
		$$
		
	\item $\overset{\vartriangleright}{\Pi_b}$ is strong $(2, 2)$:
		$$
		\|\overset{\vartriangleright}{\Pi_b}f\|_{2\rightarrow 2} \lesssim \|b\|_{BMO^\cD} \|f\|_{L^2}.
		$$
		
	\item $\overset{\vartriangleright}{\Pi_b}$ is weak $(1, 1)$:
		$$
		| \{x\in\bR^n: \overset{\vartriangleright}{\Pi_b}f(x) > \alpha\} | \leq \frac{C_n}{\alpha} \|f\|_1.
		$$
	\end{enumerate}
\end{prop}

\begin{proof}
\textit{i.} Let $x\in\bR^n$. Then 
	$$
	\Pi_bf(x) = \sum_{Q\in\cD} (b, h_Q) \La f\Ra_Q h_Q(x) = \sum_{k\in\mathbb{Z}} (b, h_{Q_k}) \La f\Ra_{Q_k} h_{Q_k}(x),
	$$
where for every $k\in\mathbb{Z}$, $Q_k$ is the unique cube in $\cD$ with side length $2^k$ that contains $x$. Fix $m\in\mathbb{Z}$:
	$$
	\left| \sum_{k>m} (b, h_{Q_k}) \La f\Ra_{Q_k} h_{Q_k}(x) \right| =
	\left| \sum_{Q \supsetneq Q_m} (b, h_Q) \La f\Ra_Q h_Q(x) \right| \leq \overset{\vartriangleright}{\Pi_b} f(x).
	$$
Taking $m\rightarrow -\infty$ finishes the proof.

\vspace{0.2in}
\textit{ii.} Let $P\in\cD$ and define
	$
	F_P(x) := \sum_{Q\supsetneq P} (b, h_Q) \La f\Ra_Q h_Q(x).
	$
	
\begin{minipage}{0.6\textwidth}
If $x\in P$, then
	$
	|F_P(x)| 
	= |\La \Pi_bf\Ra_P| \unit_P(x),
	$
so
	$$
	|F_P(x)| \leq \La|\Pi_bf|\Ra_P \unit_P(x) \leq M^\cD \Pi_b f(x).
	$$
	
	If $x \notin P$, then there is a unique $k\geq 0$ such that 
	$$x\in P^{(k+1)}\setminus P^{(k)}.$$
 So, there is a unique 
$$P_0\in \left(P^{(k+1)}\right)_{(1)}, \:\:\: P_0 \neq P^{(k)},$$ 
such that $x\in P_0$. Then:
\end{minipage}
\hfill
\noindent\begin{minipage}{0.4\textwidth}
\includegraphics[width=\linewidth]{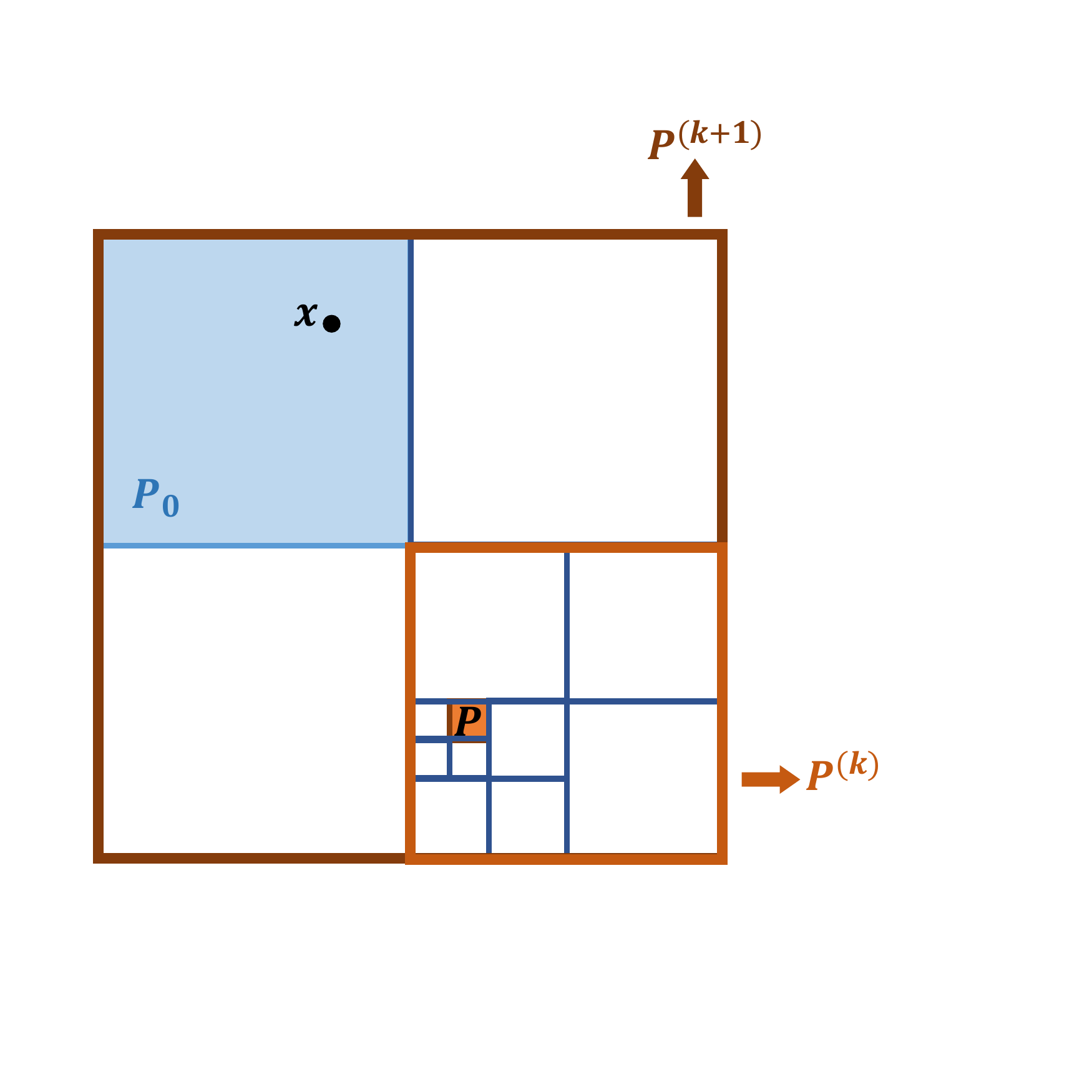}
\end{minipage}

\begin{eqnarray*}
F_P(x) &=& (b, h_{P^{(k+1)}}) \La f\Ra_{P^{(k+1)}} h_{P^{(k+1)}}(x) + \sum_{Q \supsetneq P^{(k+1)}} (b, h_Q) \La f\Ra_Q 
	\underbrace{h_Q(P^{(k+1)})}_{= h_Q(P_0)}\\
 &=& [\sum_{Q\supsetneq P_0} (b, h_Q) \La f\Ra_Q h_Q(P_0)]\unit_{P_0}(x) \\
 &=& \La \Pi_b f\Ra_{P_0} \unit_{P_0}(x),
\end{eqnarray*}
so once again $|F_P(x)|\leq M^\cD \Pi_bf(x)$. This therefore holds for all $x\in\bR^n$ and all $P\in\cD$, which proves \textit{ii}.

\vspace{0.2in}
\textit{iii.} This follows immediately from \textit{ii} and the well-known bound for $\Pi_b$ in the unweighted case:
	 $$
	 \|\overset{\vartriangleright}{\Pi_b}\|_{L^2} \leq \|M^\cD \Pi_b f\|_{L^2} \lesssim \|\Pi_b f\|_{L^2} \lesssim \|b\|_{BMO^\cD} \|f\|_{L^2}.
	 $$

\vspace{0.2in}
\textit{iv.} Once we verify $\supp(\overset{\vartriangleright}{\Pi_b}) \subset Q$ for all $Q\in\cD$, we use \textit{iii} and Proposition \ref{P:Pereyra} to conclude \textit{iv}.
	\begin{eqnarray*}
	\overset{\vartriangleright}{\Pi_b} h_Q(x) &=& \sup_{P\in\cD} |\sum_{R\supsetneq P} (b, h_R) \La h_Q\Ra_Q h_R(x)| \\
		&=& \sup_{P\subsetneq Q} |\sum_{R\supsetneq P, R \subsetneq Q} (b, h_R) h_Q(R) h_R(x)|,
	\end{eqnarray*}
which is clearly $0$ if $x \notin Q$.
\end{proof}

\subsection{Proof of Theorem \ref{T:ParaDom}.}

\begin{proof}{ }
\noindent \textit{I. The BMO decomposition.} We make use of the following modification to the Calder\'{o}n-Zygmund decomposition used in \cite{BMOdecomp} to essentially reduce a weighted BMO function to a regular BMO function. Given a weight $w$ on $\bR^n$, a function $b\in BMO^\cD(w)$, a fixed dyadic cube $Q_0\in\cD$, and $\epsilon\in (0,1)$, let the collection:
	$$
	\cE := \{\text{maximal subcubes } R\subset Q_0 \text{ s.t. } \La w\Ra_R > \frac{2}{\epsilon} \La w\Ra_{Q_0}\}
	$$
and put
	$$
	E:= \bigcup_{R\in\cE} R.
	$$
This is the collection from the usual CZ-decomposition of $w$, restricted to $Q_0$, so we have
	$$
	\sum_{R\in\cE} |R| < \frac{\epsilon}{2}|Q_0|.
	$$
But instead of defining the usual ``good function'' for $w$, we let
	$$
	a := \unit_{Q_0}(x)b(x) - \sum_{R\in\cE}(b(x)-\La b\Ra_R)\unit_R(x) = \sum_{Q\subset Q_0, Q\not\subset E}(b, h_Q) h_Q.
	$$
As shown in \cite{BMOdecomp}, this function is in unweighted BMO, with:
	$$
	a \in BMO^\cD; \:\: \|a\|_{BMO^\cD} \leq \frac{4}{\epsilon} \La w\Ra_{Q_0} \|b\|_{BMO^\cD(w)}.
	$$
Moreover,
	$$
	\forall Q\in\cD(Q_0),\: Q\not\subset E: \La a\Ra_Q = \La b\Ra_Q \text{ and } (a, h_Q)=(b, h_Q),
	$$
so whenever dealing with a cube $Q\not\subset E$, we can replace any average or Haar coefficient of $b$ -- the function in weighted BMO -- with the average or Haar coefficient of $a$ -- the function in \textit{unweighted} BMO. This has many advantages, since any usage of inequalities involving $a$ will not add any extra $A_p$ characteristics. For instance, we can use the well-known bound for Haar coefficients of BMO functions (resulting from applying the John-Nirenberg theorem to replace the $L^1$ norm in the BMO definition with the $L^2$ norm):
	$$
	|(a, h_Q)|\lesssim \sqrt{|Q|} \|a\|_{BMO^\cD}.
	$$
It also allows us to use the results on $\overset{\vartriangleright}{\Pi_a}f$ from the previous section.
	
\vspace{0.2in}

\noindent \textit{II. Use the properties of the maximal truncation of \underline{unweighted} BMO paraproducts.} We claim that there exists a constant $C_0$, depending on the dimension $n$ and on $\epsilon$, such that the set:
	$$
	F:= \{x\in Q_0: \overset{\vartriangleright}{\Pi_a}f(x) > C_0 \|a\|_{BMO^\cD}\La |f|\Ra_{Q_0}\} 
	\cup \{x\in Q_0: M_{Q_0}^\cD f(x) > C_0 \La |f|\Ra_{Q_0}\}
	$$
satisfies 
	$$
	|F| < \frac{\epsilon}{2}|Q_0|,
	$$
where $M_{Q_0}^\cD$ denotes the dyadic maximal function restricted to $Q_0$, i.e. $M_{Q_0}^\cD f(x) = \sup_{Q\subset Q_0} \La |f|\Ra_{Q_0}\unit_{Q}(x)$.
Let then the collection
	$$
	\cF := \{\text{maximal subcubes of } Q_0 \text{ contained in F}\}.
	$$

First use the well-known weak $(1,1)$ inequality for the dyadic maximal function:
	$$
	|\{x\in \bR^n: M^\cD \varphi (x) > \alpha\}| \leq \frac{C_1(n)}{\alpha}\|\varphi\|_1,
	$$
applied to $\varphi = f\unit_{Q_0}$. For all $x\in Q_0$, $M^\cD (f\unit_{Q_0})(x) = M^\cD_{Q_0}f(x)$, so
	$$
	|\{x\in Q_0: M_{Q_0}^\cD f(x) > C_0 \La |f|\Ra_{Q_0}\}| \leq \frac{C_1}{C_0}|Q_0|.
	$$

Since $a\in BMO^\cD$ we can apply the weak $(1,1)$ inequality for $\overset{\vartriangleright}{\Pi_a}$ according to Proposition \ref{P:MaxTrunc}:
	$$
	|\{x\in\bR^n: \overset{\vartriangleright}{\Pi_a} \varphi(x) > \alpha\}| \leq \frac{C_2(n)}{\alpha} \|a\|_{BMO^\cD} \|\varphi\|_1,
	$$
and let again $\varphi = f \unit_{Q_0}$. By the definition of $a$, in this case, $\Pi_a f$ sums only over $Q\subset Q_0$, so regardless of $x$ we have 
$\Pi_a \varphi = \Pi_a(f\unit_{Q_0})$. Same holds for $\overset{\vartriangleright}{\Pi_a}$:
	$$
	\overset{\vartriangleright}{\Pi_a} \varphi (x) = \sup_{P\in\cD} |\sum_{Q\supsetneq P} (a, h_Q) \La \varphi\Ra_Q h_Q(x)|
	= \sup_{P\subsetneq Q_0} |\sum_{Q\supsetneq P, Q\subset Q_0, Q\not\subset E} (b, h_Q) \La\varphi\Ra_Q h_Q(x)|,
	$$
so
	\begin{eqnarray*}
	|\{x\in Q_0: \overset{\vartriangleright}{\Pi_a} f(x) > C_0 \|a\|_{BMO^\cD} \La|f|\Ra_{Q_0}\}| &=&
		|\{x\in Q_0: \overset{\vartriangleright}{\Pi_a}(f\unit_{Q_0})(x) > C_0 \|a\|_{BMO^\cD} \La|f|\Ra_{Q_0}\}|\\
		&\leq& |\{x\in \bR^n: \overset{\vartriangleright}{\Pi_a}(f\unit_{Q_0})(x) > C_0 \|a\|_{BMO^\cD}\La|f|\Ra_{Q_0}\}|\\
		&\leq& \frac{C_2}{C_0\|a\|_{BMO^\cD}\La|f|\Ra_{Q_0}} \|a\|_{BMO^\cD} \|f\unit_{Q_0}\|_1 = \frac{C_2}{C_0}|Q_0|.
	\end{eqnarray*}
Then, as we wished,
	$$
	|F| \leq \frac{C_1+C_2}{C_0}|Q_0| < \frac{\epsilon}{2}|Q_0|,
	$$
if we choose $C_0$ large enough:
	$$
	C_0 = \frac{C(n)}{\epsilon}.
	$$

Join the collections $\cE$ and $\cF$ into:
	$$
	\cG:= \{\text{maximal subcubes of } Q_0 \text{ contained in } E\cup F\},
	$$
which then satisfies
	\begin{equation}
	\label{E:Pib-ch}
	\bigg|\bigcup_{R\in\cG}R\bigg| < \epsilon |Q_0|
	\end{equation}
	
We show that:
	\begin{equation}
	\label{E:Pib-rec}
	\unit_{Q_0}(x) \big|\Pi_{b, Q_0}f(x)\big| \leq 2C_0 \|a\|_{BMO^\cD} \La|f|\Ra_{Q_0} \unit_{Q_0}(x) +
	\sum_{R\in\cG} \unit_R(x) \big|\Pi_{b, R}f(x)\big|.
	\end{equation}
Since $\|a\|_{BMO^\cD} \leq \frac{4}{\epsilon} \La w\Ra_{Q_0}\|b\|_{BMO^\cD(w)}$, this yields
	$$
	\unit_{Q_0}(x) \big|\Pi_{b, Q_0}f(x)\big| \lesssim \frac{C_0}{\epsilon} \La w\Ra_{Q_0} \|b\|_{BMO^\cD(w)}\La |f|\Ra_{Q_0} \unit_{Q_0}(x) + 
	\sum_{R\in\cG} \unit_R(x) \big|\Pi_{b, R}f(x)\big|.
	$$
Once we have this, we recurse on the terms of the second sum, and repeat the argument: for each $R\in\cG$ construct a disjoint collection $\{R'\}\subset R$ satisfying
$|\cup R'| < \epsilon |R|$ and 
	$$
	\unit_R|\Pi_{b, R}f(x)| \lesssim \frac{C_0}{\epsilon} \La w\Ra_{R} \|b\|_{BMO^\cD(w)} \La|f|\Ra_R\unit_R(x) + \sum_{R'} |\Pi_{b, R'}f(x)|.
	$$
So we construct the collection $\cS(Q_0)$ recursively, starting with $Q_0$ as its first element, its $\cS$-children are $\cG$ and so on. We have
	$$
	|\Pi_{b, Q_0}f(x)| \lesssim \frac{C_0}{\epsilon} \|b\|_{BMO^\cD(w)} \underbrace{\sum_{Q\in\cS(Q_0)}\La w\Ra_Q \La|f|\Ra_Q \unit_Q(x)}_{ = \cA_{\cS(Q_0)}^w|f|(x)}.
	$$
Recall that $C_0 \sim \frac{C(n)}{\epsilon}$:
	$$
	|\Pi_{b, Q_0}f(x)| \lesssim \frac{C(n)}{\epsilon^2} \|b\|_{BMO^\cD(w)} \cA_{\cS(Q_0)}^w |f|(x).
	$$
The collection $\cS(Q_0)$ satisfies the $\cS$-children definition of sparse collections:
	$$
	\sum_{P\in \text{ch}_\cS(Q)}|P| < \epsilon |Q|, \forall Q\in \cS(Q_0),
	$$
so $\cS(Q_0)$ is $\frac{1}{1-\epsilon}$-Carleson. So we choose $\epsilon = \frac{\Lambda-1}{\Lambda}$
and we have the desired sparse collection with Carleson constant $\Lambda$ such that
	$$
	|\Pi_{b, Q_0}f(x)| \leq C(n) \left(\frac{\Lambda}{\Lambda-1}\right)^2 \|b\|_{BMO^\cD(w)}\cA_{\cS(Q_0)}^w|f|(x).
	$$

\vspace{0.2in}
\noindent \textit{III. Proof of \eqref{E:Pib-rec}}. We start by noting that
	\begin{eqnarray*}
	\Pi_{b, Q_0}f(x) &=& \sum_{P\subset Q_0} (b, h_P) \La f\Ra_P h_P(x)\\
	&=& \underbrace{\sum_{P\subset Q_0, P\not\subset E}(b, h_P) \La f\Ra_P h_P(x)}_{\Pi_a f(x)} +
		\sum_{R\in\cE} \underbrace{\sum_{P\subset R} (b, h_P) \La f\Ra_P h_P(x)}_{\Pi_{b,R} f(x)},
	\end{eqnarray*}
so we may decompose $\Pi_{b, Q_0}f$ as
	$$
	\unit_{Q_0}(x) \Pi_{b, Q_0}f(x) = \Pi_a f(x) + \sum_{R\in \cE} \Pi_{b, R}f(x).
	$$
Now, we have to account for the relationship to the collection $\cF$ and its union $F$.

\vspace{0.1in}
\noindent \textit{\underline{Case 1}:} $x\not\in F$.

In this case, $\overset{\vartriangleright}{\Pi_a} f(x) \leq C_0 \|a\|_{BMO^\cD} \La|f|\Ra_{Q_0}$, and since $\overset{\vartriangleright}{\Pi_a}$ dominates $\Pi_a$:
	$$
	|\Pi_a f(x)| \leq \overset{\vartriangleright}{\Pi_a}f(x) \leq C_0 \|a\|_{BMO^\cD} \La|f|\Ra_{Q_0},
	$$
so we have
	$$
	|\Pi_{b, Q_0}f(x)| \leq C_0 \|a\|_{BMO^\cD} \La|f|\Ra_{Q_0} + |\sum_{R\in\cE} \Pi_{b, R}f(x)|.
	$$
	\begin{itemize}
	\item \textit{\underline{Case 1a}:} If $x\in E$, there is a \textit{unique} $R_0 \in \cE$ such that $x\in R_0$. But then $R_0\in\cG$: say $R_0\not\in\cG$; since $R_0\subset E$, it must have been absorbed by a larger $R\supsetneq R_0$, $R\in\cF$. Then $R_0\subset R\subset F$, which contradicts $x\not\in F$. So then
	$$
	\sum_{R\in \cE} \Pi_{b, R}f(x) = \Pi_{b, R_0}f(x),
	$$
and
	$$
	|\Pi_{b, Q_0}f(x)| \leq C_0 \|a\|_{BMO^\cD} \La|f|\Ra_{Q_0} + |\Pi_{b, R_0}f(x)|, \:\:\: R_0\in \cG,
	$$
which gives \eqref{E:Pib-rec} in this case.
	
	\item \textit{\underline{Case 1b}:} If $x\not\in E$, then the second part of the sum is $0$ and we are done, having simply $|\Pi_{b, Q_0}f(x)| \leq C_0 \|a\|_{BMO^\cD} \La|f|\Ra_{Q_0}$.
	\end{itemize}

\begin{center}
\includegraphics[scale=0.5]{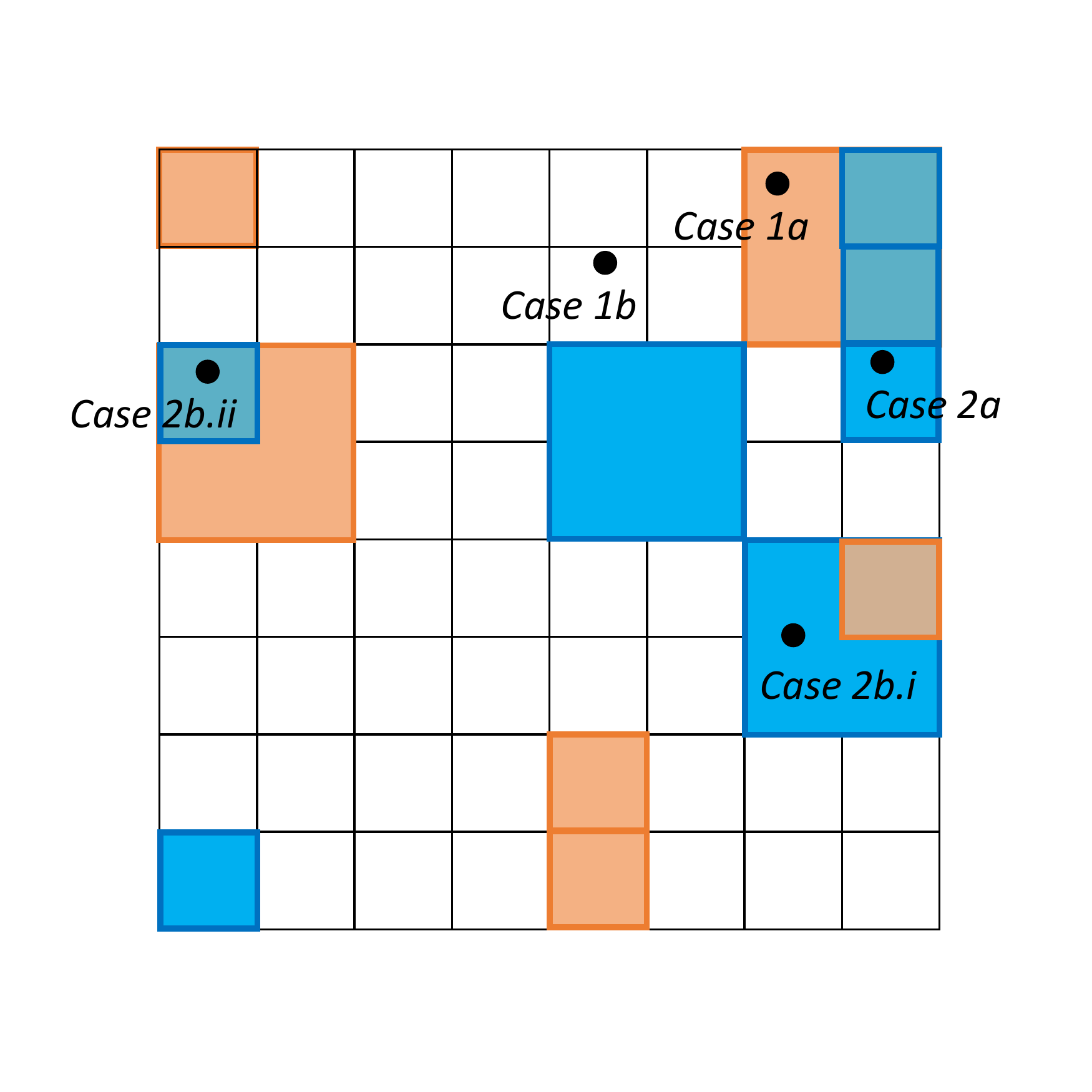}
\end{center}

\noindent \textit{\underline{Case 2}:} $x\in F$.

Then there is a \textit{unique} $P\in\cF$ such that $x\in P$. Look first at the term $\Pi_af(x) = \sum_{Q\subset Q_0} (a, h_Q)\La f\Ra_Q h_Q(x)$. Since $x\in P$, this can be expressed as
	$$
	\Pi_a f(x) = \sum_{Q\supset\hat{P}}(a, h_Q)\La f\Ra_Q h_Q(x) + \sum_{Q\subset P} (a, h_Q)\La f\Ra_Q h_Q(x),
	$$
where $\hat{P}$ denotes the dyadic parent of $P$. The first term we split into two:
	$$
	|\sum_{Q\supset\hat{P}}(a, h_Q)\La f\Ra_Q h_Q(x)| \leq
	\underbrace{|\sum_{Q\supsetneq \hat{P}} (a, h_Q) \La f\Ra_Q h_Q(x)|}_{=:A(x)} +
	\underbrace{|(a, h_{\hat{P}}) \La f\Ra_{\hat{P}} h_{\hat{P}}(x)|}_{=:B}.
	$$
\begin{itemize}
\item The term $A$ is constant on $\hat{P}$, so if $A(x) > C_0 \|a\|_{BMO^\cD}\La|f|\Ra_{Q_0}$, then $A(y) > C_0 \|a\|_{BMO^\cD}\La|f|\Ra_{Q_0}$ for all $y\in\hat{P}$. This would force $\overset{\vartriangleright}{\Pi_a} f(y) > C_0 \|a\|_{BMO^\cD}\La|f|\Ra_{Q_0}$ for all $y\in\hat{P}$, so $\hat{P} \subset F$ -- but this contradicts maximality of $P$ in $\cF$. Therefore 
	$$
	A \leq C_0 \|a\|_{BMO^\cD}\La|f|\Ra_{Q_0}.
	$$
\item Let us now look at the term $B$. If $\hat{P}\subset E$, then $B=0$. Otherwise, since $|(a, h_Q)| \lesssim \sqrt{|Q|}\|a\|_{BMO^\cD}$,
	$$
	B \leq \sqrt{|\hat{P}|} \|a\|_{BMO^\cD} \La|f|\Ra_{\hat{P}} \frac{1}{\sqrt{|\hat{P}|}} = 
	\|a\|_{BMO^\cD} \La|f|\Ra_{\hat{P}}, 
	$$
but $\La|f|\Ra_{\hat{P}} \leq C_0 \La|f|\Ra_{Q_0}$ -- otherwise, $M_{Q_0}^\cD f(y) > C_0 \La|f|\Ra_{Q_0}$ for all $y\in\hat{P}$, which would force $\hat{P}\subset F$, again contradicting maximality of $P$ in $\cF$.
\end{itemize}

So
	$$
	|\sum_{Q\supset\hat{P}}(a, h_Q)\La f\Ra_Q h_Q(x)| \lesssim C_0 \|a\|_{BMO^\cD} \La|f|\Ra_{Q_0},
	$$
giving us that
	$$
	|\Pi_{b, Q_0}f(x)| \lesssim C_0 \|a\|_{BMO^\cD} \La|f|\Ra_{Q_0}  + |C|,
	$$
where the term $C$ is defined as
	$$
	C := \sum_{Q\subset P, Q\not\subset E} (b, h_Q) \La f\Ra_Q h_Q(x) + \sum_{R\in\cE} \Pi_{b, R} f(x).
	$$
We claim that
	$$
	C = \Pi_{b, R_0}f(x),
	$$
where $R_0$ is the unique element of $\cG$ such that $x\in R_0$:
	\begin{itemize}
	\item \textit{\underline{Case 2a}:} If $P\cap E = \emptyset$, then $R_0 = P$ and $C = \Pi_{b,P}f(x) = \Pi_{b, R_0}f(x)$ (the first term is $\Pi_{b, P}$ and the second term is $0$).
	\item \textit{\underline{Case 2b}:} If $P\cap E \neq \emptyset$:
		\begin{itemize}
		\item \textit{\underline{Case 2b.i}:} If $P$ contains some elements of $\cE$, then again $R_0 = P$ and we can ``fill in the blanks'' in the first term with the $\Pi_{b,R}$'s from the second term:
			$$
			C = \bigg[\sum_{Q\subset P, Q\not\subset E} (b, h_Q) \La f\Ra_Q h_Q(x) + \sum_{R\in\cE, R\subset P} \Pi_{b, R} f(x)\bigg] + 
			\underbrace{\sum_{R\in\cE, R\not\subset P} \Pi_{b, R}f(x)}_{=0}
			= \Pi_{b,P}f(x) = \Pi_{b, R_0}f(x).
			$$
		\item \textit{\underline{Case 2b.ii}:} If $P\subset S_0$ for some $S_0\in\cE$, then $R_0 = S_0$ and the first term in $C$ is $0$ (because $P\subset E$), and the second term is 
		$\sum_{R\in\cE} \Pi_{b,R}f(x) = \Pi_{b, S_0}f(x) = \Pi_{b, R_0}f(x)$.
		\end{itemize}
	\end{itemize}
This concludes the proof.
\end{proof}

\begin{remark}
One can also use Theorem \ref{T:ParaDom} to obtain a full $\bR^n$ domination, losing the requirement for no infinitely increasing chains. Say $f$ is such that $\supp(f)\subset Q_0$ for some $Q_0\in\cD$ (or, for general compactly supported functions, $\supp(f)$ is contained in at most $2^n$ disjoint $Q_k\in\cD$). Then
	$$
	\Pi_b f(x) = \Pi_{b, Q_0}f(x) + \bigg(\sum_{Q\supsetneq Q_0} (b, h_Q)\frac{1}{|Q|}h_Q(x)\bigg)\int_{Q_0}f.
	$$
Note that, as an application of the modified CZ-decomposition used in Part I of the proof above, one can obtain
	$$
	|(b, h_Q)| \lesssim \sqrt{|Q|}\La w\Ra_Q \|b\|_{BMO^\cD(w)}, \:\:\forall Q\in\cD, b \in BMO^\cD(w).
	$$
To see this, let $Q\in\cD$ and apply the decomposition to $b$ over $Q$:
	$$
	\cE:=\{\text{maximal subcubes } R\subset Q_0 \text{ s.t. } \La w\Ra_R>2\La w\Ra_{Q}\}; \:\:\: E:=\bigcup_{R\in\cE}R;
	$$
	$$
	a:= \sum_{R\subset Q; R\not\subset E}(b, h_Q)h_Q \in BMO^\cD \text{ with } \|a\|_{BMO^\cD} \leq 4\La w\Ra_Q \|b\|_{BMO^\cD(w)}.
	$$
Since $Q$ itself is not selected for $\cE$, $Q\not\subset E$, so $(a, h_Q)=(b, h_Q)$. Finally, then:
	$$
	|(b, h_Q)| = |(a, h_Q)|\lesssim \sqrt{|Q|}\|a\|_{BMO^\cD}\leq \sqrt{|Q|}4\La w\Ra_Q \|b\|_{BMO^\cD(w)}.
	$$

Returning to $\Pi_b f$, suppose first that $x\not\in Q_0$. Then there is a unique $k\geq 1$ such that $x\in Q_0^{(k)}\setminus Q_0^{(k-1)}$, and
	$$
	\Pi_b f(x) = \bigg(\sum_{Q\supset Q_0^{(k)}} (b, h_Q)\frac{1}{|Q|}h_Q(x)\bigg) \int_{Q_0}f.
	$$
Then
	\begin{eqnarray*}
	|\Pi_b f(x)| &\leq& \sum_{Q\supset Q_0^{(k)}} |(b, h_Q)|\frac{1}{|Q|}\frac{1}{\sqrt{|Q|}} \big(\int_{Q_0}|f|\big)\\
	&\lesssim& \sum_{Q\supset Q_0^{(k)}} \La w\Ra_Q \|b\|_{BMO^\cD(w)} \frac{1}{|Q|}\int_{Q_0}|f|\\
	&=& \|b\|_{BMO^\cD(w)} \sum_{Q\supset Q_0^{(k)}} \La w\Ra_Q \La|f|\Ra_Q.
	\end{eqnarray*}
If, on the other hand, $x\in Q_0$,
	$$
	\Pi_b f(x) = \Pi_{b, Q_0}f(x) + \sum_{Q\supsetneq Q_0}(b, h_Q)\frac{1}{|Q|}h_Q(Q_0)\int_{Q_0}f,
	$$
so
	$$
	|\Pi_b f(x)|\lesssim |\Pi_{b, Q_0}f(x)| + \|b\|_{BMO^\cD(w)}\sum_{Q\supsetneq Q_0} \La w\Ra_Q \La|f|\Ra_Q.
	$$
By Theorem \ref{T:ParaDom}, there is a $\Lambda$-Carleson sparse collection $\cS(Q_0)$ such that
	$$
	|\Pi_{b, Q_0}f(x)| \leq C(n) \bigg(\frac{\Lambda}{\Lambda-1}\bigg)^2 \|b\|_{BMO^\cD(w)} 
	\cA^w_{\cS(Q_0)}|f|(x).
	$$
So form a sparse collection $\cS$ as follows:
	$$
	\cS := \cS(Q_0)\cup \bigcup_{k=1}^\infty Q_0^{(k)},
	$$
with $Q_0^{(k-1)}$ being the only $\cS$-child of $Q_0^{(k)}$ for all $k\geq 1$. It is easy to see that $\cS$ is $(\Lambda+1)$-Carleson. Moreover the associated sparse operator
	$$
	\cA_\cS^w f = \cA_{\cS(Q_0)}^w f + \sum_{Q\supsetneq Q_0}\La w\Ra_Q\La f\Ra_Q\unit_{Q}
	$$
appears exactly in the previous inequalities, which can be expressed as:
	$$
	x\not\in Q_0: |\Pi_b f(x)| \lesssim \|b\|_{BMO^\cD(w)}\cA_\cS^w|f|(x);
	$$
	$$
	x\in Q_0: |\Pi_b f(x)| \lesssim C(n) \bigg(\frac{\Lambda}{\Lambda-1}\bigg)^2 \|b\|_{BMO^\cD(w)} 
	\cA^w_{\cS}|f|(x).
	$$
So indeed
	$$
	|\Pi_b f(x)| \lesssim \|b\|_{BMO^\cD(w)} \cA_\cS^w |f|(x), \:\:\forall x\in \bR^n,
	$$
for all compactly supported $f$.
\end{remark}

\begin{remark}
If we let $f\equiv 1$ in $\Pi_{b, Q_0}f$, we have
	$$
	\Pi_{b, Q_0}1(x) = \sum_{Q\subset Q_0} (b, h_Q)h_Q(x) = (b(x) - \La b\Ra_{Q_0})\unit_{Q_0}(x).
	$$
So, applying Theorem \ref{T:ParaDom} to the function $f\equiv 1$ essentially gives us that local mean oscillations of functions in $BMO^\cD(w)$ can be dominated by one of the sparse BMO functions in Section \ref{Ss:SparseBMO}:
\end{remark}

\begin{corollary}
There is a dimensional constant $C(n)$ such that for all $\Lambda>1$, weights $w$ on $\bR^n$, $b\in BMO^\cD(w)$ and $Q_0\in\cD$, there is a $\Lambda$-Carleson sparse collection $\cS(Q_0)\subset\cD(Q_0)$ such that
	\begin{eqnarray*}
	|(b(x) - \La b\Ra_{Q_0})\unit_{Q_0}(x)| &\leq& C(n)\left(\frac{\Lambda}{\Lambda-1}\right)^2 \|b\|_{BMO^\cD(w)}\sum_{Q\in\cS(Q_0)}\La w\Ra_Q\unit_Q(x)\\
	&=& C(n)\left(\frac{\Lambda}{\Lambda-1}\right)^2 \|b\|_{BMO^\cD(w)} b^w_{\cS(Q_0)}(x).
	\end{eqnarray*}
\end{corollary}

\appendix
\section{Proof of Theorem \ref{T:SparseBMO}}
\label{A:my-sparse}
Recall that we are given $\cS\in\Upsilon^\cD(\bR^n)$ and the associated function
	$$
	b_\cS := \sum_{Q\in\cS}\unit_Q,
	$$
and we wish to show that
	$$
	\|b_\cS\|_{BMO^\cD} \leq \Lambda,
	$$
where $\Lambda$ is the Carleson constant of $\cS$.

\begin {proof}
Let $Q_0\in\cD$ be fixed. We wish to estimate $\frac{1}{|Q_0|} \int_{Q_0}|b-\La b\Ra_{Q_0}|\,dx$, and recall that
	$$
	(b_\cS-\La b_\cS\Ra_{Q_0}) \unit_{Q_0} = \sum_{Q\in\cS, Q\subsetneq Q_0}\unit_Q - (\tau_\cS)_{Q_0}\unit_{Q_0}, \text{ where }
	(\tau_\cS)_P := \frac{1}{|P|}\sum_{Q\in\cS, Q\subsetneq P}|Q| \leq\Lambda \:\:\forall P\in\cD.
	$$
In fact, 
	$$ \text{If } P\in\cS, \text{ then } (\tau_\cS)_P \leq \Lambda-1.
	$$
	
With $Q_0\in\cD$ fixed, here we are only looking at $\cS(Q_0):=\{Q\in\cS: Q\subset Q_0\}$. We define the collections as sets:
	\begin{eqnarray*}
	\cS_1 &:=& \text{ch}_\cS(Q_0) \text{ (the }\cS\text{-children of }Q_0\text{) and } S_1:= \bigcup_{Q_1\in\cS_1}Q_1;\\
	\cS_2 &:=& \{Q_2\in\text{ch}_\cS(Q_1): Q_1\in\cS_1\} \text{ and } S_2:=\bigcup_{Q_2\in\cS_2}Q_2,
	\end{eqnarray*}
so $\cS_2$ are the ``$\cS$-grandchildren'' of $Q_0$, the second generation of $\cS$-cubes in $Q_0$. Generally,
	$$
	\cS_k := \{Q_k\in\text{ch}_\cS(Q_{k-1}) : Q_{k-1}\in\cS_{k-1}\} \text{ and } S_k:=\bigcup_{Q_k\in\cS_k}Q_k.
	$$
Note that:
	\begin{itemize}
	\item Each $S_k$ is a disjoint union of $Q_k\in\cS_k$, as each $\cS_k$ is a pairwise disjoint collection.
	\item The sets $S_k$ satisfy $Q_0\supset S_1\supset S_2\supset\ldots$
	\item Moreover
		$$
		\left|\bigcap_{k=1}^\infty S_k\right|=0,
		$$
	since $\bigcap_{k=1}^\infty$ is exactly the set of all $x$ contained in infinitely many elements of $\cS(Q_0)$. We can also see this directly, as the series $\sum_{k=1}^\infty |S_k| \leq \Lambda |Q_0|$ converges.
	\end{itemize}

For ease of notation, denote for now
	$$
	\theta:= (\tau_\cS)_{Q_0} = \frac{1}{|Q_0|} \sum_{Q\in\cS, Q\subsetneq Q_0}|Q| \leq\Lambda.
	$$
We have:
	\begin{eqnarray*}
	\frac{1}{|Q_0|}\int_{Q_0}|b_\cS-\La b_\cS\Ra_{Q_0}|\,dx &=& \frac{1}{|Q_0|} \int_{Q_0} \bigg|\sum_{Q\in\cS,Q\subsetneq Q_0}\unit_Q(x) - \theta\bigg|\,dx\\
	&=& \frac{1}{|Q_0|} \int_{Q_0\setminus S_1} |\theta|\,dx + \frac{1}{|Q_0|}\int_{S_1} \bigg|\sum_{Q\in\cS,Q\subsetneq Q_0}\unit_Q(x) - \theta\bigg|\,dx
	\end{eqnarray*}
Since $S_1$ is a disjoint union of $Q_1\in\cS_1$:
	\begin{eqnarray*}
	\frac{1}{|Q_0|} \int_{S_1} \bigg|\sum_{Q\in\cS,Q\subsetneq Q_0}\unit_Q(x) - \theta\bigg|\,dx &=&
		\frac{1}{|Q_0|}\sum_{Q_1\in\cS_1} \int_{Q_1}\bigg|\sum_{Q\in\cS,Q\subset Q_1}\unit_Q(x) - \theta\bigg|\,dx\\
	&=& \frac{1}{|Q_0|} \sum_{Q_1\in\cS_1} \int_{Q_1} \bigg|\sum_{Q\in\cS,Q\subsetneq Q_1}\unit_Q(x) +1 - \theta\bigg|\,dx\\
	&=& \frac{1}{|Q_0|} \bigg[ \sum_{Q_1\in\cS_1}\bigg( \int_{Q_1\setminus S_2}|1-\theta|\,dx +\sum_{\substack{Q_2\in\cS_2\\ Q_2\subsetneq Q_1}}\int_{Q_2}
		\big|\sum_{\substack{Q\in\cS\\ Q\subset Q_2}} \unit_{Q}(x) +1-\theta\big|\,dx \bigg) \bigg]\\
	&=& \frac{1}{|Q_0|}|1-\theta \underbrace{|\sum_{Q_1\in\cS_1}|Q_1\setminus S_2|}_{|S_1\setminus S_2|} + \frac{1}{|Q_0|}\sum_{Q_2\in\cS_2}\int_{Q_2}
		\big|\sum_{\substack{Q\in\cS\\ Q\subsetneq Q_2}} \unit_{Q}(x) +2-\theta\big|\,dx.
	\end{eqnarray*}
So
	$$
	\frac{1}{|Q_0|}\int_{Q_0}|b_\cS-\La b_\cS\Ra_{Q_0}|\,dx =\theta\frac{|Q_0\setminus S_1|}{|Q_0|} + |1-\theta|\frac{|S_1\setminus S_2|}{|Q_0|}
		+ \frac{1}{|Q_0|}\sum_{Q_2\in\cS_2}\int_{Q_2}
		\big|\sum_{\substack{Q\in\cS\\ Q\subsetneq Q_2}} \unit_{Q}(x) +2-\theta\big|\,dx.
	$$
We can apply the same reasoning to each $Q_2\in\cS_2$:
	$$
	\int_{Q_2} \big|\sum_{\substack{Q\in\cS\\ Q\subsetneq Q_2}} \unit_{Q}(x) +2-\theta\big|\,dx = \int_{Q_2\setminus S_3} |2-\theta|\,dx
		+\sum_{\substack{Q_3\in\cS_3\\ Q_3\subsetneq Q_2}} \int_{Q_3} \big|\sum_{\substack{Q\in\cS\\ Q\subsetneq Q_3}} \unit_{Q}(x) +3-\theta\big|\,dx,
	$$
and we can conclude inductively
	\begin{equation}
	\label{E:SpBMO1}
	\frac{1}{|Q_0|}\int_{Q_0}|b_\cS-\La b_\cS\Ra_{Q_0}|\,dx  = \theta\frac{|Q_0\setminus S_1|}{|Q_0|} + |1-\theta|\frac{|S_1\setminus S_2|}{|Q_0|}
	+ |2-\theta|\frac{|S_2\setminus S_3|}{|Q_0|}+\ldots
	\end{equation}

Suppose for a moment that $\theta\leq 1$. Then \eqref{E:SpBMO1} becomes
	\begin{eqnarray*}
	&& \theta\frac{|Q_0\setminus S_1|}{|Q_0|} + (1-\theta)\frac{|S_1\setminus S_2|}{|Q_0|}
	+ (2-\theta)\frac{|S_2\setminus S_3|}{|Q_0|}+\ldots\\
	&=& \frac{1}{|Q_0|} \bigg(\theta|Q_0\setminus S_1| +(1-\theta)|S_1|-(1-\theta)|S_2|+(2-\theta)|S_2|-(2-\theta)|S_3|+(3-\theta)|S_3|-\ldots \bigg)\\
	&=& \frac{1}{|Q_0|} \bigg(\theta|Q_0\setminus S_1| + (1-\theta)|S_1| +|S_2|+|S_3|+\ldots\bigg).
	\end{eqnarray*}

\begin{remark}
Thoroughly, we have above a sequence of partial sums
	\begin{eqnarray*}
	a_k &=& c|S_1\setminus S_2| + (c+1)|S_2\setminus S_3|+\ldots + (c+k-1)|S_k\setminus S_{k+1}|\\
	&=& c|S_1|-c|S_2|+(c+1)|S_2|-(c+1)|S_3|+\ldots +(c+k-1)|S_k|-(c+k-1)|S_{k+1}|\\
	&=& c|S_1|+|S_2|+|S_3|+\ldots+|S_k|-(c+k-1)|S_{k+1}|,\\
	\end{eqnarray*}
where $c=(1-\theta)\geq 0$. We know that:
	\begin{itemize}
	\item The series $\sum_{k=1}^\infty |S_k|$ converges, by the Carleson property;
	\item The ``remainder'' $(c+k-1)|S_{k+1}|\rightarrow 0$ as $k\rightarrow\infty$, because the series $\sum_{k=1}^\infty k|S_k|$ also converges:
		\begin{eqnarray*}
		\sum_{k=1}^\infty k|S_k| &=& |S_1|+2|S_2|+3|S_3|+\ldots\\
			&=& |S_1|+|S_2|+|S_3|+\ldots \leq \Lambda|S_1|\\
			&+&|S_2|+|S_3|+\ldots \leq\Lambda|S_2|\\
			&+& |S_3|+|S_4|+\ldots \leq \Lambda|S_3|\\
			&+&\vdots\\
			&\leq& \Lambda(|S_1|+|S_2|+\ldots) \leq \Lambda^2|S_1|.
		\end{eqnarray*}
	\end{itemize}
So 
	$$\lim_{k\rightarrow\infty} a_k = c|S_1|+\sum_{k=2}^\infty |S_k| - \underbrace{\lim_{k\rightarrow\infty}(c+k-1)|S_{k+1}|}_{=0}
	$$
and 
	$$
	\frac{1}{|Q_0|}\int_{Q_0}|b_\cS-\La b_\cS\Ra_{Q_0}|\,dx = \frac{1}{|Q_0|} \bigg(\theta|Q_0\setminus S_1| + (1-\theta)|S_1| +|S_2|+|S_3|+\ldots\bigg)
	$$
holds.
\end{remark}

Now,
	$$
	|S_2|+|S_3|+\ldots = \sum_{Q_1\in\cS_1}\underbrace{\bigg(\sum_{Q\in\cS, Q\subsetneq Q_1}|Q|\bigg)}_{\leq (\Lambda-1)|Q_1|\text{ because }Q_1\in\cS}
	\leq (\Lambda-1)\sum_{Q_1\in\cS_1}|Q_1|=(\Lambda-1)|S_1|.
	$$
So
	\begin{eqnarray*}
	\frac{1}{|Q_0|}\int_{Q_0}|b_\cS-\La b_\cS\Ra_{Q_0}|\,dx &\leq& \frac{1}{|Q_0|} \bigg(
		\theta|Q_0\setminus S_1| +(1-\theta)|S_1| + (\Lambda-1)|S_1|\bigg)\\
		&=& \frac{1}{|Q_0|}\bigg(\theta|Q_0\setminus S_1| +(\Lambda-\theta)|S_1|\bigg)\\
		&\leq& \frac{1}{|Q_0|}\bigg(\Lambda|Q_0\setminus S_1| +\Lambda|S_1|\bigg)\\ 
		&=& \frac{\Lambda}{|Q_0|}\big(|Q_0\setminus S_1|+|S_1|\big)\\
		&=& \Lambda.
	\end{eqnarray*}

Generally, if $n < \theta \leq (n+1)$ for some $n\in\mathbb{N}$: the right hand side of \eqref{E:SpBMO1} becomes
$$
\frac{1}{|Q_0|}\bigg[
\theta|Q_0\setminus S_1|+(\theta-1)|S_1\setminus S_2| +\ldots+(\theta-n)|S_n\setminus S_{n+1}|
+\underbrace{(n+1-\theta)|S_{n+1}\setminus S_{n+2}|+(n+2-\theta)|S_{n+2}\setminus S_{n+3}|+\ldots}_{(n+1-\theta)|S_{n+1}| +
\underbrace{|S_{n+2}|+|S_{n+3}|+\ldots}_{\leq(\Lambda-1)|S_{n+1}|}}
\bigg]
$$
\begin{eqnarray*}
&&\leq \frac{1}{|Q_0|}\bigg[\theta|Q_0\setminus S_1|+(\theta-1)|S_1\setminus S_2|+\ldots+(\theta-n)|S_n\setminus S_{n+1}| + (\Lambda+n-\theta)|S_{n+1}| \bigg]\\
&&\leq \frac{1}{|Q_0|}\bigg[\Lambda|Q_0\setminus S_1|+\Lambda|S_1\setminus S_2|+\ldots + \Lambda|S_n\setminus S_{n+1}|+\Lambda|S_{n+1}| \bigg]\\
&&\leq \frac{\Lambda}{|Q_0|}\bigg(|Q_0\setminus S_1| +|S_1\setminus S_2|+\ldots+|S_n\setminus S_{n+1}|+|S_{n+1}|\bigg)\\
&&= \Lambda.
\end{eqnarray*}
\end{proof}

\section{Proof of Theorem \ref{T:ParaComp}}
\label{A:ParaComp}

Say we have $a\in BMO^\cD(\bR^n)$, $b\in BMO^\cD(w)$ where $w$ is a weight on $\bR^n$, and a fixed $Q_0\in\cD$. We look at
	$$
	\Pi_a^*\Pi_{b, Q_0}f := \sum_{Q\subset Q_0} (a, h_Q)(b, h_Q) \La f\Ra_Q\frac{\unit_Q}{|Q|}
	$$
and the inner product
	$$
	(\Pi_a^*\Pi_{b, Q_0}f,g)=\sum_{Q\subset Q_0}(a,h_Q)(b,h_Q)\La f\Ra_Q\La g\Ra_Q.
	$$
Within $Q_0$ we form the local CZ-decompositions of $f$ and $g$, and the BMO decomposition of $b$:
	\begin{eqnarray*}
	\cE_1 &:=& \{\text{maximal subcubes } R\in\cD(Q_0) \text{ s.t. } \La|f|\Ra_R > \frac{3}{\epsilon}\La|f|\Ra_{Q_0}\}; \:\: E_1:=\cup_{R\in\cE_1}R;\\
	\cE_2 &:=& \{\text{maximal subcubes } R\in\cD(Q_0) \text{ s.t. } \La|g|\Ra_R > \frac{3}{\epsilon}\La|g|\Ra_{Q_0}\}; \:\: E_2:=\cup_{R\in\cE_2}R;\\
	\cE_3 &:=& \{\text{maximal subcubes } R\in\cD(Q_0) \text{ s.t. } \La w\Ra_R > \frac{3}{\epsilon}\La w\Ra_{Q_0}\}; \:\:\:\: E_3:=\cup_{R\in\cE_3}R.
	\end{eqnarray*}
Based on $\cE_3$ we define 
	$$
	\widetilde{b} := \unit_{Q_0}b-\sum_{R\in\cE_3}(b-\La b\Ra_R)\unit_R=\sum_{\substack{Q\subset Q_0\\ Q\not\subset E_3}}(b, h_Q)h_Q,
	$$
which satisfies $\widetilde{b}\in BMO^\cD(\bR^n)$ with
	$$
	\|\widetilde{b}\|_{BMO^\cD} \leq \frac{6}{\epsilon}\La w\Ra_{Q_0}\|b\|_{BMO^\cD(w)}.
	$$
Moreover, $(b, h_Q) = (\widetilde{b},h_Q)$ for all $Q\subset Q_0$, $Q\not\subset E_3$. Each collection $\cE_i$ satisfies 
	$$
	\sum_{R\in\cE_i}|R|\leq \frac{\epsilon}{3}|Q_0|.
	$$
Finally, let
	$$
	E :=E_1\cup E_2\cup E_3 \text{ and } \cE:=\{\text{maximal subcubes }R\in\cD(Q_0)\text{ s.t. }R\subset E\}.
	$$
Then
	$$
	\sum_{R\in\cE}|R|\leq\epsilon|Q_0|.
	$$

Now look at $(\Pi_a^*\Pi_{b, Q_0}f,g)$ and split the sum as
	\begin{equation}
	\label{E:App2}
	|(\Pi_a^*\Pi_{b, Q_0}f,g)| \leq \sum_{\substack{Q\subset Q_0\\ Q\not\subset E}}|(a, h_Q)|\:|(b, h_Q)|\:\La|f|\Ra_Q\La|g|\Ra_Q
	+ \sum_{R\in\cE}|(\Pi_a^*\Pi_{b, R}f,g)|.
	\end{equation}
For every $Q\subset Q_0$, $Q\not\subset E$, we have:
	$$
	\La|f|\Ra_Q\leq \frac{3}{\epsilon}\La|f|\Ra_{Q_0}, \:\: \La|g|\Ra_Q\leq \frac{3}{\epsilon}\La|g|\Ra_{Q_0}, \text{ and } (b, h_Q) = (\widetilde{b},h_Q),
	$$
so:
	\begin{eqnarray*}
	\sum_{\substack{Q\subset Q_0\\ Q\not\subset E}}|(a, h_Q)|\:|(b, h_Q)|\:\La|f|\Ra_Q\La|g|\Ra_Q 
	&\leq& \frac{9}{\epsilon^2}\La|f|\Ra_{Q_0}\La|g|\Ra_{Q_0}\sum_{Q\subset Q_0, Q\not\subset E}|(a, h_Q)|\:|(\widetilde{b}, h_Q)|\\
	&\leq& \frac{9}{\epsilon^2}\La|f|\Ra_{Q_0}\La|g|\Ra_{Q_0} 
		\underbrace{\bigg(\sum_{Q\subset Q_0} |(a, h_Q)|^2\bigg)^{1/2}}_{\leq C(n)\sqrt{|Q_0|}\|a\|_{BMO^\cD}}
		\underbrace{\bigg(\sum_{Q\subset Q_0} |(\widetilde{b}, h_Q)|^2\bigg)^{1/2}}_{\substack{\leq C(n)\sqrt{|Q_0|}\|\widetilde{b}\|_{BMO^\cD} \\ 
		\leq C(n)\frac{6}{\epsilon} \La w\Ra_{Q_0} \sqrt{|Q_0|} \|b\|_{BMO^\cD(w)}}},
	\end{eqnarray*}
where $C(n)$ is the dimensional constant arising from using the John-Nirenberg Theorem.
Finally, we have
	$$
	\sum_{\substack{Q\subset Q_0\\ Q\not\subset E}}|(a, h_Q)|\:|(b, h_Q)|\:\La|f|\Ra_Q\La|g|\Ra_Q  \leq 
	\frac{C(n)}{\epsilon^3}\|a\|_{BMO^\cD}\|b\|_{BMO^\cD(w)}\La|f|\Ra_{Q_0}\La|g|\Ra_{Q_0}\La w\Ra_{Q_0}|Q_0|.
	$$
	
Now we recurse on the $\sum_{R\in\cE}$ terms in \eqref{E:App2} and form $\cS(Q_0)$ by adding $Q_0$ first, $\cE$ are the $\cS$-children of $Q_0$, and so on. The collection $\cS(Q_0)$ satisfies the $\cS$-children definition of sparseness, with $\sum_{R\in\text{ch}_\cS(Q)}|R|\leq\epsilon|Q|$ for all $Q\in\cS(Q_0)$, so it is $\frac{1}{1-\epsilon}$-Carleson. So, if we choose $\epsilon=\frac{\Lambda}{\Lambda-1}$, we have
	$$
	\bigg|\sum_{Q\subset Q_0} (a,h_Q)(b,h_Q)\La f\Ra_Q\La g\Ra_Q\bigg| \leq C(n) \left(\frac{\Lambda}{\Lambda-1}\right)^3 \|a\|_{BMO^\cD}\|b\|_{BMO^\cD(w)} 
	\underbrace{\sum_{Q\in\cS(Q_0)} \La w\Ra_Q \La|f|\Ra_Q\La|g|\Ra_Q|Q|}_{=(\cA_{\cS(Q_0)}^w|f|, |g|)}
	$$	
	
We summarize this below:

\begin{prop}
\label{P:App2}
There is a dimensional constant $C(n)$ such that for all $a\in BMO^\cD$, $b\in BMO^\cD(w)$, where $w$ is a weight on $\bR^n$, fixed $Q_0\in\cD$ and $\Lambda>1$, there is a $\Lambda$-Carleson sparse collection $\cS(Q_0)\subset\cD(Q_0)$ such that
	$$
	\bigg|\sum_{Q\subset Q_0} (a,h_Q)(b,h_Q)\La f\Ra_Q\La g\Ra_Q\bigg| \leq C(n) \left(\frac{\Lambda}{\Lambda-1}\right)^3 \|a\|_{BMO^\cD}\|b\|_{BMO^\cD(w)}
	(\cA_{\cS(Q_0)}^w|f|, |g|).
	$$
\end{prop}

\vspace{0.1in}
\begin{center}
$\ast$
\end{center}
\vspace{0.1in}

Say now we have Bloom weights $\mu,\lambda \in A_p$ ($1<p<\infty$), $\nu:=\mu^{1/p}\lambda^{-1/p}$ on $\bR^n$ and $a\in BMO^\cD$, $b\in BMO^\cD(\nu)$. Suppose further that $a$ has finite Haar expansion. Then there are at most $2^n$ disjoint dyadic cubes $Q_k\in\cD$, $1\leq k \leq 2^n$, such that
$a=\sum_k\sum_{Q\subset Q_k}(a,h_Q)h_Q$, and then
	$$
	(\Pi^*_a\Pi_b f, g) = \sum_k (\Pi_a^*\Pi_{b, Q_k}f, g)
	$$
Given $\Lambda>1$, by Proposition \ref{P:App2}, there is for each $k$ a $\Lambda$-Carleson sparse collection $\cS(Q_k)\subset\cD(Q_k)$ such that
	$$
	\bigg| (\Pi_a^*\Pi_{b, Q_k}f, g) \bigg| \leq C(n) \left(\frac{\Lambda}{\Lambda-1}\right)^3 \|a\|_{BMO^\cD}\|b\|_{BMO^\cD(\nu)}
	(\cA_{\cS(Q_k)}^\nu|f|, |g|).
	$$
Then
	$$
	\bigg| (\Pi_a^*\Pi_{b}f, g) \bigg| \leq C(n) \left(\frac{\Lambda}{\Lambda-1}\right)^3 \|a\|_{BMO^\cD}\|b\|_{BMO^\cD(\nu)}
	(\cA_\cS^\nu|f|, |g|),	
	$$
where $\cS:=\cup_k\cS(Q_k)$ is a $\Lambda$-Carleson sparse collection in $\Upsilon^\cD(\bR^n)$.

Take now $f\in L^p(\mu)$ and $g\in L^{p'}(\lambda')$. By a simple application of H\"{o}lder's inequality:
	$$
	|(\cA_\cS^\nu|f|,|g|)| \leq \|\cA_\cS^\nu:L^p(\mu)\rightarrow L^p(\lambda)\|\: \|f\|_{L^p(\mu)}\|g\|_{L^{p'}(\lambda')}.
	$$
Then
	$$
	\|\Pi_a^*\Pi_b : L^p(\mu)\rightarrow L^p(\lambda)\| \leq C(n)\|a\|_{BMO^\cD}\|b\|_{BMO^\cD(\nu)} 
	\sup_{\substack{\cS\in\Upsilon^\cD(\bR^n)\\ \Lambda_{(\cS)}=\Lambda}} \left(\frac{\Lambda}{\Lambda-1}\right)^3
	\|\cA_\cS^\nu:L^p(\mu)\rightarrow L^p(\lambda)\|
	$$
holds for all $a$ with finite Haar expansion, and therefore for all $a$. This proves Theorem \ref{T:ParaComp}.

\newpage

\end{document}